\documentclass[10pt]{amsart}
\usepackage{amssymb}
\usepackage{bm}
\usepackage{graphicx}
\usepackage[centertags]{amsmath}
\usepackage{amsfonts}
\usepackage{amsthm}
\usepackage{amsbsy}
\usepackage{mathtools}
\usepackage{mathrsfs}
\usepackage{cases}
\usepackage[all]{xy}
\usepackage{algorithm2e}
\usepackage[linktocpage=true]{hyperref}
\usepackage{hyperref}
\hypersetup{
    colorlinks=true,
    linkcolor=blue,
    filecolor=blue,
    citecolor=blue,
    urlcolor=cyan}
\usepackage[margin=4cm]{geometry}
\newtheorem{thm}{Theorem}[section]
\newtheorem{cor}[thm]{Corollary}
\newtheorem{lem}[thm]{Lemma}
\newtheorem{sbl}[thm]{Sublemma}
\newtheorem{prop}[thm]{Proposition}
\newtheorem{claim}[thm]{Claim}
\newtheorem{fact}[thm]{Fact}
\newtheorem{observation}[thm]{Observation}

\newtheorem{defn}[thm]{Definition}

\theoremstyle{remark}
\newtheorem{rem}[thm]{Remark}

\newcommand{\rr}{\mathbb{R}}

\newcommand{\ee}{\varepsilon}

\newcommand{\meg}{\geqslant}
\newcommand{\mik}{\leqslant}

\newcommand{\thtensor}{\boldsymbol{\theta}=\langle \theta_i: i\in [n]^d\rangle}

\newcommand{\Fcal}{\mathcal{F}}
\newcommand{\Gcal}{\mathcal{G}}
\newcommand{\Scal}{\mathcal{S}}

\newcommand{\Acal}{\mathcal{A}}

\newcommand{\Dcal}{\mathcal{D}}
\newcommand{\Xcal}{\mathcal{X}}
\newcommand{\Ycal}{\mathcal{Y}}
\newcommand{\Forb}[2]{\mathrm{Forbid}\left(#1,#2\right)}
\newcommand{\pmeasure}[1]{\ensuremath{\mu_{p}\left(#1\right)}}

\newcommand{\seminorm}[1]{{\left\vert\kern-0.25ex\left\vert\kern-0.25ex\left\vert #1
    \right\vert\kern-0.25ex\right\vert\kern-0.25ex\right\vert}}
\begin{document}

\title[Forbidden sparse intersections]{Forbidden sparse intersections}

\author{Pandelis Dodos and Miltiadis Karamanlis}

\address{Department of Mathematics, University of Athens, Panepistimiopolis 157 84, Athens, Greece}
\email{pdodos@math.uoa.gr}

\address{Department of Mathematics, University of Athens, Panepistimiopolis 157 84, Athens, Greece}
\email{kararemilt@gmail.com}



\begin{abstract}
Let $n$ be a positive integer, let $0<p\leqslant p'\leqslant \frac12$, and let $\ell\leqslant pn$ be a nonnegative integer.
We prove that if $\mathcal{F},\mathcal{G}\subseteq \{0,1\}^n$ are two families whose cross intersections forbid
$\ell$---that is, they satisfy $|A\cap B|\neq \ell$ for every $A\in\mathcal{F}$ and every $B\in\mathcal{G}$---then,
setting $t\coloneqq \min\{\ell,pn-\ell\}$, we have the subgaussian bound
\[ \mu_p(\mathcal{F})\, \mu_{p'}(\mathcal{G}) \leqslant 2\exp\Big( - \frac{t^2}{58^2\,pn}\Big), \]
where $\mu_p$ and $\mu_{p'}$ denote the $p$-biased and $p'$-biased measures on $\{0,1\}^n$ respectively.
\end{abstract}

\maketitle

\tableofcontents


\section{Introduction} \label{sec1}

\numberwithin{equation}{section}

Extremal set theory can be traced back to the seminal work\footnote{Although published in 1961,
the Erd\H{o}s--Ko--Rado theorem was actually discovered much earlier, in~1938.} of Erd\H{o}s, Ko and
Rado \cite{EKR61} who obtained sharp estimates of the cardinality of a family $\mathcal{A}\subseteq \binom{[n]}{k}$
that is \textit{intersecting}, that is, it satisfies $A\cap B\neq \emptyset$ for every $A,B\in\mathcal{A}$.
(Here, and in the rest of this paper, $\binom{[n]}{k}$ denotes the set of all $k$-element subsets of
the discrete interval $[n]\coloneqq \{1,\dots,n\}$.) Since then, it is an active subfield of combinatorics;
we refer the reader to \cite{El22,FT18} for recent expositions of this theory and its applications.

\subsection{The Erd\H{o}s--S\'{o}s problem and related results} \label{subsec1.1}

A more challenging problem was posed in 1971 by Erd\H{o}s and S\'{o}s (see \cite{Erd75}); it asks to determine,
for a given triple $\ell\mik k\mik n$ of positive integers, the cardinality of the largest family
$\mathcal{A}\subseteq \binom{[n]}{k}$ whose intersections \textit{forbid} $\ell$, that is, it satisfies
$|A\cap B|\neq \ell$ for every $A,B\in\mathcal{A}$. Early significant progress on the Erd\H{o}s--S\'{o}s problem
was made by Frankl--Wilson \cite{FW81} and \text{Frankl--F\"{u}redi} \cite{FF85}. Somewhat later, in 1987,
a breakthrough was achieved by Frankl and R\"{o}dl who obtained exponential estimates in the regime where
$\ell$ is proportional~to~$n$. More precisely, it is shown in \cite[Theorem 1.5]{FR87} that for
every $0<\ee<p\mik \frac12$ there exists a constant $\gamma(\ee,p)>0$ such that for every pair of
positive integers $\ell,n$ with $\ee n\mik \ell\mik pn-\ee n$, and every pair of families\footnote{We
identify every $A\subseteq [n]$ with its indicator function $\mathbf{1}_A\in \{0,1\}^n$.}
$\Fcal,\Gcal\subseteq \{0,1\}^n$ whose cross intersections forbid $\ell$---that is, they satisfy
$|A\cap B|\neq \ell$ for every $A\in\Fcal$ and every $B\in\Gcal$---we have
\begin{equation} \label{e1.1}
\mu_p(\Fcal)\, \mu_p(\Gcal) \mik \big(1-\gamma(\ee,p)\big)^n,
\end{equation}
where $\mu_p$ denotes the $p$-biased measure on $\{0,1\}^n$, namely, the probability measure
on $\{0,1\}^n$ defined by setting for every $A\subseteq [n]$,
\begin{equation} \label{e1.2}
\mu_p\big(\{A\}\big) \coloneqq p^{|A|} (1-p)^{n-|A|}.
\end{equation}
(For the case $\Fcal=\Gcal$, an alternative proof was given by Keevash and Long \cite{KLo16}; see also \cite{KSZ22}.)
The work of Frankl and R\"{o}dl has proven to be very influential, and it has found applications
in a number of different areas such as discrete geometry \cite{FR90},
communication complexity~\cite{S99} and quantum computing \cite{BCW99}.

Further progress on the Erd\H{o}s--S\'{o}s problem was made by several authors, including the very recent
works of Ellis--Keller--Lifshitz \cite{EKL24}, Keller--Lifshitz \cite{KLi21} and \text{Kupavskii--Zaharov} \cite{KZ24}
(see, also, \cite{KLLM23,KLo20} for closely related developments). Collectively, the papers \cite{EKL24,KLi21}
obtain the sharp estimate $|\mathcal{A}|\mik \binom{n-(\ell+1)}{k-(\ell+1)}$ for every family
$\mathcal{A}\subseteq \binom{[n]}{k}$ whose intersections forbid $\ell$ in the regime
$2\ell< k\mik \big(\frac12-\ee\big)n$ with $n\meg n_0(\ell,\ee)$ for some (unspecified)
threshold function $n_0(\ell,\ee)$. The more recent paper \cite{KZ24} extends this sharp estimate in the regime
$\ell=\lceil k^\beta\rceil$ and $n=\lceil k^\alpha\rceil$, where $\alpha,\beta>0$ are positive reals
with $\beta<\frac12$ and $\alpha>1+2\beta$ and $k$ is sufficiently large in terms of $\alpha,\beta$.

\subsection{The main estimate} \label{subsec1.2}

Our main result provides new estimates of the product of the biased measures of a pair of families
with forbidden cross intersections and, in particular, it bridges the gap between the aforementioned results.
\begin{thm} \label{main}
Let $n$ be a positive integer, let $0<p\mik p'\mik \frac12$, and let $\ell\mik pn$ be a nonnegative integer.
Also let $\Fcal,\Gcal\subseteq \{0,1\}^n$ be two families whose cross intersections forbid $\ell$, that is,
$|A\cap B|\neq \ell$ for every $A\in\Fcal$ and $B\in\Gcal$. Then, setting $t\coloneqq \min\{\ell,pn-\ell\}$,
we have the subgaussian bound
\begin{equation} \label{e1.main}
\mu_p(\Fcal)\, \mu_{p'}(\Gcal) \mik 2\exp\Big( - \frac{t^2}{58^2\, pn}\Big).
\end{equation}
\end{thm}
Note that Theorem \ref{main} extends the Frankl--R\"{o}dl theorem \cite[Theorem 1.5]{FR87};
indeed, the bound \eqref{e1.main} is nontrivial if $80\sqrt{pn}\mik \ell \mik pn-80\sqrt{pn}$
uniformly for $p\meg \frac{160^2}{n}$.
\begin{rem} \label{r1.2}
Theorem \ref{main} can be extended to cover the case of all parameters $p,p'$ in the regime
$0<p\mik p'\mik 1-p$, and it also has a version for pairs of families $\Fcal,\Gcal$ that are contained
in two, possibly different, layers of the cube. We present these (standard) extensions in Section \ref{sec7}.
\end{rem}
\begin{rem} \label{r1.3}
The subgaussian bound \eqref{e1.main} is actually optimal, modulo universal constants,
for various choices of $p,p'$ and $\ell$. We discuss these issues in Section \ref{sec8}.
\end{rem}
\begin{rem} \label{r1.4}
Theorem \ref{main} also has a supersaturation version, which is in the spirit of
\cite[Theorem 1.14]{FR87}; see Section \ref{sec9} for details.
\end{rem}

\subsection{Outline of the proof of Theorem \ref{main}} \label{subsec1.3}

The proof of Theorem \ref{main} follows the general strategy invented by Frankl and R\"{o}dl \cite{FR87}.
The idea is to gradually ``deform" the given families, and eventually arrive at a pair of families whose
cross intersections forbid an initial or a final interval; the measures of these final families can then
be estimated by standard probabilistic tools. This ``deformation" is entirely algorithmic,
and it is the heart of the proof.

The algorithm in \cite{FR87} takes as an input two families $\Fcal,\Gcal\subseteq \{0,1\}^n$
whose cross intersections forbid an interval of $[n]$, and it starts by using a density increment
argument in order to show that the two sections
\[ \Fcal_0\coloneqq \{A: n\notin A\in\Fcal\} \ \ \ \text{ and } \ \ \
\Fcal_1\coloneqq \big\{A\setminus\{n\}: n\in A\in\mathcal{F}\big\} \]
of $\Fcal$, have roughly the same measure. Once this is done, the algorithm proceeds by comparing the measures
of the union $\Gcal_0\cup \Gcal_1$ and the intersection $\Gcal_0\cap \Gcal_1$ of the sections of $\Gcal$.
Again, a density increment argument is used to ensure that the measures of $\Gcal_0\cup \Gcal_1$ and
$\Gcal_0\cap \Gcal_1$ are roughly equal, which in turn implies that the two sections, $\Gcal_0$ and
$\Gcal_1$, are almost equal. One can then use this structural information to produce a pair
of families whose cross intersections forbid a larger interval, while at the same time one keeps
control of the product of their biased measures.

While the algorithm of Frankl and R\"{o}dl is elegant and efficient, unfortunately it leads
to suboptimal results as $p$ gets smaller, and it hits a barrier\footnote{It is also not clear
if the algorithm works if $\min\{\ell,pn-\ell\}=o(pn)$ and $p=\Theta(1)$, but the obstacles in this regime
seem somewhat less serious.} at $p=o(1)$. The reason is rather simple: if $p$ is small,
then having (or not having) density increment for the section $\mathcal{A}_1$ of a family
$\mathcal{A}\subseteq \{0,1\}^n$ has negligible effect on the measure of the other section $\mathcal{A}_0$.

We resolve this issue by introducing a new algorithm that also takes as an input two families
$\Fcal,\Gcal\subseteq \{0,1\}^n$ whose cross intersections forbid an interval of $[n]$, and it starts by seeking
for a density increment for one of the pairs $(\Fcal_1,\Gcal_1)$, $(\Fcal_0, \Gcal_0\cup \Gcal_1)$
and $(\Fcal_0\cup\Fcal_1, \Gcal_0)$. However, the density increment the algorithm is searching for,
is not \textit{uniform} and depends on the specific pair it is looking at (as well as the parameter $p$).
The particular choice of the density increments is justified analytically: if the algorithm does not
succeed in this search, then this yields a strong \textit{lower} bound for the product of the measures
of $\Fcal_1$ and $\Gcal_0\cap \Gcal_1$, or the product of the measures of $\Fcal_0\cap\Fcal_1$
and $\Gcal_1$. This is the content of Lemma~\ref{l4.1} (the ``widening lemma") in Section \ref{sec4}.
With this information at hand, we may proceed as in the algorithm of Frankl and R\"{o}dl.
The main novelty (and technical difficulty) of the proof of Theorem \ref{main} is thus to show
that this rough outline is actually feasible by appropriately selecting the various parameters.


\section{Background material} \label{sec2}

\subsection{General notation} \label{subsec2.1}

For every pair $a,b$ of integers with $0\mik a\mik b$ by $[a,b]$ we denote the discrete
interval $\{k\in\mathbb{Z}: a\mik k \mik b\}$. Also recall that for every positive integer~$n$
and every nonnegative integer $k\mik n$, we set $[n]\coloneqq \{1,\dots,n\}$ and
$\binom{[n]}{k}\coloneqq\{A\subseteq [n]: |A|=k\}$; moreover, for every
$t\meg 0$, we set $[n]^{\mik t}\coloneqq \{A\subseteq [n]: |A|\mik t\}$,
$[n]^{<t}\coloneqq \{A\subseteq [n]: |A|<t\}$,
$[n]^{\meg t}\coloneqq \{A\subseteq [n]: |A|\meg t\}$ and
$[n]^{>t}\coloneqq \{A\subseteq [n]: |A|>t\}$.

\subsection{Families of sets} \label{subsec2.2}

Let $n$ be a positive integer, and let $\Acal\subseteq \{0,1\}^n$.
We say that $\Acal$ is \textit{downwards closed} if for every
$A\in\Acal$ and every $B\subseteq A$ we have that $B\in\Acal$; respectively, we say that
$\Acal$ is \textit{upwards closed} if for every $A\in\Acal$ and every $B\supseteq A$
we have that $B\in\Acal$. If, in addition, $n\meg 2$, then we set
\begin{equation} \label{e2.1}
\mathcal{A}_0\coloneqq \{A: n\notin A\in\mathcal{A}\} \ \ \ \text{ and } \ \ \
\mathcal{A}_1\coloneqq \big\{A\setminus \{n\}: n\in A\in\mathcal{A}\big\},
\end{equation}
and we view both $\Acal_0$ and $\Acal_1$ as subfamilies of $\{0,1\}^{n-1}$.

It is also convenient to introduce the following definition.
\begin{defn}[Forbidden intersections] \label{d2.1}
Let $n$ be a positive integer, let $L\subseteq [n]$, and let $\Fcal,\Gcal\subseteq \{0,1\}^n$.
We write $(\Fcal,\Gcal)\in \mathrm{Forbid}(n,L)$ to denote the fact that the cross intersections
of $\Fcal$ and\, $\Gcal$ forbid\, $L$, that is, $|A\cap B|\notin L$ for every $A\in\Fcal$ and $B\in\Gcal$.
\end{defn}

\subsection{$p$-biased measures} \label{subsec2.3}

Recall that for every positive integer $n$ and every $0<p<1$ by $\mu_p$ we
denote the $p$-biased probability measure on $\{0,1\}^n$ defined in \eqref{e1.2}.
We record, for future use, the following elementary property of these measures.
\begin{fact} \label{f2.2}
Let $n$ be a positive integer, let $0<p\mik p'<1$, and let $\Acal\subseteq \{0,1\}^n$.
If~$\Acal$ is upwards closed, then $\mu_p(\Acal)\mik \mu_{p'}(\Acal)$; respectively,
if $\Acal$ is downwards closed, then $\mu_p(\Acal) \meg \mu_{p'}(\Acal)$.
\end{fact}

\subsection{Chernoff bounds} \label{subsec2.4}

We will need the following standard estimates of the biased measure of the tails of the
binomial distribution (see, \textit{e.g.}, \cite[Appendix A]{AS16}).
\begin{lem} \label{l2.3}
Let $n$ be a positive integer, let $t$ be a nonnegative real, and let\, $0<p<1$.
\begin{enumerate}
\item[(i)] If\, $p\mik \frac12 $ and $pn\mik t\mik 2pn$, then
\begin{equation} \label{e2.2}
\mu_p\left([n]^{\meg t}\right)
\mik \exp\left(-\frac{(t-pn)^2}{6p(1-p)n}\right).
\end{equation}
\item[(ii)] If\, $p\meg \frac12 $ and $t\meg pn$, then
\begin{equation} \label{e2.3}
\mu_p\left([n]^{\meg t}\right) \mik \exp\left(-\frac{(t-pn)^2}{2p(1-p)n}\right).
\end{equation}
\end{enumerate}
\end{lem}

\subsection{Estimates of binomial coefficients} \label{subsec2.5}

We will also need the following basic estimates of binomial coefficients that follow from a
non-asymptotic version of Stirling's approximation---see, \textit{e.g.}, \cite{Ro55}---and elementary computations.
\begin{fact} \label{f2.4}
Let $n\meg 2$ be an integer, and let $k\in [n-1]$. Then we have
\begin{equation} \label{e2.4}
\frac{24}{25 \, \sqrt{2\pi}} \cdot \sqrt{\frac{n}{k(n-k)}}\cdot \frac{n^n}{k^k(n-k)^{n-k}}<
\binom{n}{k}< \frac{1}{\sqrt{2\pi}}\cdot \sqrt{\frac{n}{k(n-k)}}\cdot \frac{n^n}{k^k(n-k)^{n-k}};
\end{equation}
in particular, for every $\Fcal\subseteq \binom{[n]}{k}$ we have
\begin{equation} \label{e2.5}
\sqrt{2\pi}\cdot \sqrt{\frac{k(n-k)}{n}} \cdot \mu_{\frac{k}{n}}(\Fcal) <
\frac{|\Fcal|}{\binom{n}{k}} < \frac{25 \, \sqrt{2\pi}}{24} \cdot \sqrt{\frac{k(n-k)}{n}}\cdot \mu_{\frac{k}{n}}(\Fcal).
\end{equation}
Moreover, if $H\colon [0,1]\to\rr$ denotes the binary entropy function\footnote{Recall that $H(0)=H(1)=0$,
and $H(x)=-x\log_2(x)-(1-x)\log_2(1-x)$ if $0<x<1$.}, then
\begin{equation} \label{e2.6}
\frac{24}{25 \, \sqrt{2\pi}} \cdot \sqrt{\frac{n}{k(n-k)}}\cdot 2^{nH(\frac{k}{n})} \mik \binom{n}{k} \mik
\frac{1}{\sqrt{2\pi}}\cdot \sqrt{\frac{n}{k(n-k)}}\cdot 2^{nH(\frac{k}{n})}.
\end{equation}
\end{fact}


\section{Forbidding initial or final intervals} \label{sec?}

\numberwithin{equation}{section}

Our goal in this section is to obtain estimates for the product of the biased measures of a pair of
families $\Fcal, \Gcal\subseteq \{0,1\}^n$ whose cross intersections forbid an initial or a final subinterval of $[n]$.
This information is needed for the proof of Theorem \ref{main}.

We note that closely related problems have been studied extensively in extremal combinatorics; see
\cite{El22,FT18} and the references therein.  We shall obtain the desired estimates, however,
from the following well-known concentration inequality for the biased measures.
As usual, for a family $\Acal\subseteq \{0,1\}^n$ and a nonnegative real $t\mik n$ we set
$\Acal_t\coloneqq \{H\subseteq [n]: \exists A\in\Acal \text{ such that } |H\bigtriangleup A|\mik t\}$.
\begin{prop} \label{p3.1}
Let $n$ be a positive integer, let $0<p<1$, let $t\mik pn$ be a nonnegative real, and let
$\Acal\subseteq \{0,1\}^n$ such that $\pmeasure{\Acal}\meg \frac12$.
\begin{enumerate}
\item[(i)] If\, $0<p\mik \frac12$, then we have
\begin{equation} \label{e3.1}
\pmeasure{\Acal_t}\meg 1-\exp\left(-\frac{t^2}{6p(1-p)n}\right).
\end{equation}
\item[(ii)] If\, $\frac12<p<1$, then we have
\begin{equation} \label{e3.2}
\pmeasure{\Acal_t}\meg 1-\exp\left(-\frac{t^2}{2p(1-p)n}\right).
\end{equation}
\end{enumerate}
\end{prop}
Proposition \ref{p3.1} follows from the proof of \cite[Proposition 2.4]{BHT06}
in the work of Bobkov, Choudr\'{e} and Tetali, which in turn is based on results
of Bollob\'{a}s--Leader \cite{BL91}, Jogdeo--Samuels \cite{JS68}, and Talagrand \cite{Ta89}.
Since Proposition \ref{p3.1} is not explicitly isolated in \cite{BHT06}, for the
convenience of the reader we briefly recall the argument.
\begin{proof}[Proof of Proposition \emph{\ref{p3.1}}]
As it is mentioned in \cite{BHT06}, it is enough to prove the result under the additional assumption that
$\Acal$ is downwards closed. Indeed, Step~1 through Step 4 in the proof
of \cite[Theorem 7]{Ta89} carry out this reduction. So, suppose that $\Acal$ is downwards closed
with $\pmeasure{\Acal}\meg \frac12$. By \cite[Theorem 3.2 and Corollary 3.1]{JS68},
we~have
\begin{equation} \label{e3.3}
\pmeasure{\Acal}\meg \mu_p\big([n]^{\mik \lfloor pn \rfloor}\big);
\end{equation}
that is, the median of the binomial distribution $\mathrm{Bin}(n,p)$ is greater than or equal to $\lfloor pn \rfloor$.
Moreover, since $\Acal$ is downwards closed, by \cite[Corollary 5]{BL91}, we have
\begin{equation} \label{e3.4}
\pmeasure{\Acal_t}\meg \mu_p\big([n]^{\mik \lfloor pn \rfloor+t}\big),
\end{equation}
which in turn implies that
\begin{equation} \label{e3.5}
1-\pmeasure{\Acal_t}\mik
\mu_p\big([n]^{>\lfloor pn\rfloor +t}\big)\mik \mu_p\big([n]^{\meg pn +t}\big).
\end{equation}
Therefore, if $0<p\mik \frac12$, then \eqref{e3.1} follows from \eqref{e3.5} and \eqref{e2.2},
while if $\frac12< p<1$, then \eqref{e3.2} follows from \eqref{e3.5} and \eqref{e2.3}.
\end{proof}
Proposition \ref{p3.1} will be used in the following form (the proof is straightforward, and it is left
to the reader).
\begin{cor} \label{c3.2}
Let $n$ be a positive integer, let $0<p<1$, let $t\mik pn$ be a nonnegative real, and let
$\Acal\subseteq \{0,1\}^n$.
\begin{enumerate}
\item[(i)] If\, $0<p\mik \frac12$ and $\pmeasure{\Acal}>\exp\big(-\frac{t^2}{6p(1-p)n}\big)$,
then $\pmeasure{\Acal_{2t}}>1-\exp\big(-\frac{t^2}{6p(1-p)n}\big)$.
\item[(ii)] If\, $\frac12<p<1$ and $\pmeasure{\Acal}>\exp\big(-\frac{t^2}{2p(1-p)n}\big)$,
then $\pmeasure{\Acal_{2t}}>1-\exp\big(-\frac{t^2}{2p(1-p)n}\big)$.
\end{enumerate}
\end{cor}
We are now ready to state the main result in this section.
\begin{lem} \label{l3.3}
Let $n$ be a positive integer, let $0<p\mik p'\mik 1-p$ with $p\mik \frac12$, and let $\alpha\mik pn$ be
a positive integer. Also let $\Fcal,\Gcal\subseteq \{0,1\}^n$.
\begin{enumerate}
\item[(i)] If\, $(\Fcal,\Gcal)\in\Forb{n}{[0,\alpha]}$, then
\begin{equation}\label{e3.6}
\mu_p(\Fcal)\,\mu_{p'}(\Gcal)\mik\exp\left(-\frac{\alpha^2}{24p(1-p)n}\right).
\end{equation}
\item[(ii)] If\, $(\Fcal,\Gcal)\in\Forb{n}{[pn-\alpha,n]}$, then
\begin{equation}\label{e3.7}
\mu_{p}(\Fcal)\,\mu_{p'}(\Gcal)\mik 2\exp\left(-\frac{\alpha^2}{24p(1-p)n}\right).
\end{equation}
\end{enumerate}
\end{lem}
\begin{proof}
We start with the proof of part (i). Clearly, we may assume that the pair $(\Fcal,\Gcal)$ is
optimal, in the sense that it maximizes the quantity in the left-hand-side of \eqref{e3.6};
consequently, we may assume that $\Fcal$ and $\Gcal$ are both upwards closed.
Next observe that if $\mu_{p}(\Fcal)\mik \exp\big(-\frac{\alpha^2}{24p(1-p)n}\big)$,
then \eqref{e3.6} is straightforward. Therefore, we may also assume that
$\mu_{p}(\Fcal)>\exp\big(-\frac{\alpha^2}{24p(1-p)n}\big)$. By  Corollary \ref{c3.2} applied for ``$t=\frac{\alpha}{2}$",
we obtain that $\mu_{p}(\Fcal_{\alpha})>1-\exp\big(-\frac{\alpha^2}{24p(1-p)n}\big)$, where
$\Fcal_{\alpha}\coloneqq \{H\subseteq [n]: \exists A\in\Fcal \text{ with } |H\bigtriangleup A|\mik \alpha\}$.
Set $\overline{\Gcal}\coloneqq \{[n]\setminus G:G\in\Gcal\}$, and note that for every $F\in\Fcal$ and
every $G\in\Gcal$ we have that $|F\bigtriangleup([n]\setminus G)|>\alpha$. This yields that
$\overline{\Gcal}\cap\Fcal_{\alpha}=\emptyset$, which in turn implies that  $\mu_{p}(\overline{\Gcal})<\exp\big(-\frac{\alpha^2}{24p(1-p)n}\big)$. Finally,
since $\Gcal$ is upwards closed and $p'\mik 1-p$, by Fact~\ref{f2.2}, we conclude that
\begin{equation} \label{e3.8}
\exp\left(-\frac{\alpha^2}{24p(1-p)n}\right)> \mu_{p}(\overline{\Gcal})=
\mu_{1-p}(\Gcal) \meg \mu_{p'}(\Gcal) \meg \mu_{p}(\Fcal)\, \mu_{p'}(\Gcal).
\end{equation}

We proceed to the proof of part (ii). As before, we may assume that the pair $(\Fcal,\Gcal)$ is optimal
and, hence, that $\Fcal$ and $\Gcal$ are both downwards closed. Consequently, by Fact~\ref{f2.2},
it is enough to show that
\begin{equation} \label{e3.9}
\mu_{p}(\Fcal)\, \mu_{p}(\Gcal)\mik 2 \exp\left(-\frac{\alpha^2}{24p(1-p)n}\right).
\end{equation}
Setting
\begin{enumerate}
\item[$\bullet$] $\Fcal^{\mik pn-\frac{\alpha}{2}}\coloneqq \{A\in \Fcal: |A|\mik pn-\frac{\alpha}{2}\}$,
$\Fcal^{>pn-\frac{\alpha}{2}}\coloneqq \{A\in \Fcal : |A|>pn-\frac{\alpha}{2}\}$, and
\item[$\bullet$] $\Gcal^{\mik pn-\frac{\alpha}{2}}\coloneqq \{B\in \Gcal : |B|\mik pn-\frac{\alpha}{2}\}$,
$\Gcal^{>pn-\frac{\alpha}{2}}\coloneqq \{B\in \Gcal : |B|>pn-\frac{\alpha}{2}\}$,
\end{enumerate}
by part (ii) of Lemma \ref{l2.3} applied for $\mu_{1-p}$, we see that
\begin{equation} \label{e3.10}
\max\big\{ \mu_{p}\big(\Fcal^{\mik pn-\frac{\alpha}{2}}\big),
\mu_{p}\big(\Gcal^{\mik pn-\frac{\alpha}{2}}\big)\big\} \mik
\exp\left(-\frac{\alpha^2}{8p(1-p)n}\right);
\end{equation}
thus, if $\mu_{p}\big(\Fcal^{> pn-\frac{\alpha}{2}}\big)<\exp\big(-\frac{\alpha^2}{24p(1-p)n}\big)$,
then the result follows from \eqref{e3.10}. So, suppose that
$\mu_{p}\big(\Fcal^{>pn-\frac{\alpha}{2}}\big)\meg \exp\big(-\frac{\alpha^2}{24p(1-p)n}\big)$.
By Corollary~\ref{c3.2} again applied for ``$t=\frac{\alpha}{2}$", we obtain that
\begin{equation} \label{e3.11}
\mu_{p}\big(\Fcal^{>pn-\frac{\alpha}{2}}_\alpha\big)\meg
1-\exp\left(-\frac{\alpha^2}{24p(1-p)n}\right),
\end{equation}
where $\Fcal^{>pn-\frac{\alpha}{2}}_\alpha\coloneqq \{H\subseteq [n]:
\exists A\in \Fcal^{>pn-\frac{\alpha}{2}} \text{ with } |H\bigtriangleup A|\mik \alpha\}$.
Now note that, since $(\Fcal,\Gcal)\in\Forb{n}{[pn-\alpha,n]}$, for every
$F\in \Fcal^{> pn-\frac{\alpha}{2}}$ and every $G\in \Gcal^{> pn-\frac{\alpha}{2}}$ we have
\begin{equation} \label{e3.12}
|F\bigtriangleup G|=|F|+|G|-2|F\cap G|>\alpha.
\end{equation}
This observation yields that $\Gcal^{> pn-\frac{\alpha}{2}}\cap \Fcal^{> pn-\frac{\alpha}{2}}_\alpha=\emptyset$,
and therefore, by \eqref{e3.11}, we obtain that
$\mu_{p}\big(\Gcal^{> pn-\frac{\alpha}{2}}\big)<\exp\big(-\frac{\alpha^2}{24p(1-p)n}\big)$.
Inequality \eqref{e3.7} follows from this estimate and~\eqref{e3.10}.
The proof of Lemma \ref{l3.3} is thus completed.
\end{proof}


\section{The widening lemma} \label{sec4}

\numberwithin{equation}{section}

The section is devoted to the proof of the following lemma. (Recall that for every family
$\mathcal{A}\subseteq \{0,1\}^n$ ($n\meg 2$) by $\Acal_0$ and $\Acal_1$ we denote the
sections of $\Acal$ defined in \eqref{e2.1}; we also recall that we view
$\mathcal{A}_0$ and $\mathcal{A}_1$ as families in $\{0,1\}^{n-1}$.)
\begin{lem}[Widening lemma] \label{l4.1}
Let $n\meg 2$ be an integer, let $0<p\mik p'\mik \frac12$, and let $\Fcal,\Gcal\subseteq \{0,1\}^n$
be nonempty. Also let\, $0<\delta<\frac{1}{10}$, and assume that
\begin{equation} \label{e4.1}
\mu_p(\Fcal_1)\,\mu_{p'}(\Gcal_1)  \mik (1+\delta)\, \mu_p(\Fcal)\,\mu_{p'}(\Gcal)
\end{equation}
and
\begin{equation} \label{e4.2}  \max\big\{ \mu_p(\Fcal_0)\, \mu_{p'}(\Gcal_0  \cup\Gcal_1), \,
\mu_p(\Fcal_0\cup\Fcal_1)\, \mu_{p'}(\Gcal_0) \big\}  \mik \Big(1+\frac{p}{1-p}\delta\Big)\,
\mu_p(\Fcal)\, \mu_{p'}(\Gcal).
\end{equation}
Then we have
\begin{equation} \label{e4.3}
\max\big\{ \mu_p(\Fcal_1)\,\mu_{p'}(\Gcal_0\cap\Gcal_1), \,
\mu_p(\Fcal_0\cap\Fcal_1)\,\mu_{p'}(\Gcal_1)\big\} >
\Big(1-\delta-2\frac{p}{1-p}\delta^2\Big)\, \mu_p(\Fcal)\, \mu_{p'}(\Gcal).
\end{equation}
\end{lem}
As we have noted in Subsection \ref{subsec1.3}, Lemma \ref{l4.1} is a crucial ingredient
of the proof of Theorem \ref{main}. That said, we advise the reader to skip its proof
at first reading, and return to this section once the basic steps of the proof
of Theorem \ref{main} have been properly understood.

We shall deduce Lemma \ref{l4.1} from the following, purely analytical, result.
\begin{sbl} \label{s4.2}
Let $x,x',y,y',z,z',w,w',p,p',\delta\in[-1,1]$ be real numbers with $y,w\meg 0$, $y',w'\mik 0$,
$\max\{x,x'\}\mik w$, $\max\{z,z'\}\mik y$, $0<p\mik p'\mik \frac{1}{2}$ and $0<\delta<\frac{1}{10}$.
Assume that the following identities
\begin{align}
\label{e4.4} x'=-\frac{p}{1-p}x,  & \ \ \ \ \ \ \ \ w+w'=x+x'=\frac{1-2p}{1-p}x, \\
\label{e4.5} z'=-\frac{p'}{1-p'}z, & \ \ \ \ \ \ \ \ y+y'=z+z'=\frac{1-2p'}{1-p'}z,
\end{align}
as well as the following inequalities
\begin{align}
\label{e4.6} (1+x)(1+z)& \mik 1+\delta, \\
\label{e4.7} (1+x')(1+y) & \mik 1+\frac{p}{1-p}\delta,\\
\label{e4.8} (1+z')(1+w) & \mik 1+\frac{p}{1-p}\delta,
\end{align}
are satisfied. Then at least one of the following inequalities
\begin{align}
\label{e4.9} (1+x)(1+y') & > 1-\delta-2\frac{p}{1-p}\delta^2, \\
\label{e4.10} (1+z)(1+w')& > 1-\delta-2\frac{p}{1-p}\delta^2,
\end{align}
must also be satisfied.
\end{sbl}
We postpone the proof of Sublemma \ref{s4.2} to the end of this section. At this point, let us
give the proof of Lemma \ref{l4.1}.
\begin{proof}[Proof of Lemma \emph{\ref{l4.1}}]
Notice, first, that
\begin{equation} \label{e4.11}
\mu_p(\Fcal)=p\mu_p(\Fcal_1) +(1-p)\mu_p(\Fcal_0) \ \ \text{ and } \ \
\mu_{p'}(\Gcal)=p'\mu_{p'}(\Gcal_1) +(1-p')\mu_{p'}(\Gcal_0).
\end{equation}
Next, define the real numbers $x,x',y,y',z,z',w,w'$ by setting
\begin{align*}
\frac{\mu_p(\Fcal_1)}{\mu_p(\Fcal)}&=1+x, & \frac{\mu_{p'}(\Gcal_1)}{\mu_{p'}(\Gcal)} & = 1+z,\\
\frac{\mu_p(\Fcal_0)}{\mu_p(\Fcal)}&=1+x' ,&	\frac{\mu_{p'}(\Gcal_0)}{\mu_{p'}(\Gcal)} & =1+z',\\
\frac{\mu_p(\Fcal_0\cup\Fcal_1)}{\mu_p(\Fcal)}& = 1+w, & \frac{\mu_{p'}(\Gcal_0\cup\Gcal_1)}{\mu_{p'}(\Gcal)} & =1+y,\\
\frac{\mu_p(\Fcal_0\cap\Fcal_1)}{\pmeasure{\Fcal}}&=1+w', & \frac{\mu_{p'}(\Gcal_0\cap\Gcal_1)}{\mu_{p'}(\Gcal)} & =1+y'.\\
\end{align*}
With these choices, the result follows from Sublemma \ref{s4.2} after taking into account the identities in \eqref{e4.11}.
\end{proof}

\subsection{Proof of Sublemma \ref{s4.2}} \label{subsec4.1}

First observe that, by \eqref{e4.4} and \eqref{e4.5}, we have
\begin{equation} \label{e4.12}
xy'+x'y =\frac{1-2p'}{1-p'}xz-\frac{xy}{1-p} \ \ \ \text{ and } \ \ \
zw'+z'w=\frac{1-2p}{1-p}xz-\frac{zw}{1-p'}
\end{equation}
that yields that
\begin{align}
\label{e4.13} & (1+x)(1+y')+(1+x')(1+y)=2+\frac{1-2p}{1-p}x+\frac{1-2p'}{1-p'}z(1+x)-\frac{xy}{1-p}, \\
\label{e4.14} & (1+z)(1+w')+(1+z')(1+w)=2+\frac{1-2p}{1-p}x(1+z)+\frac{1-2p'}{1-p'}z-\frac{zw}{1-p'}.
\end{align}
Combining these equalities with \eqref{e4.7} and \eqref{e4.8}, we obtain that
\begin{align} \label{e4.15}
(1+x)(1+y') \meg 1-\frac{p}{1-p}\delta +\frac{1-2p}{1-p}x+\frac{1-2p'}{1-p'}z(1+x)-\frac{xy}{1-p}
\end{align}
and
\begin{align} \label{e4.16}
(1+z)(1+w') \meg 1-\frac{p}{1-p}\delta + \frac{1-2p}{1-p}x(1+z)+\frac{1-2p'}{1-p'}z-\frac{zw}{1-p'}.
\end{align}
Next observe that since $y$ and $w$ are nonnegative, by \eqref{e4.7} and \eqref{e4.8}, we have
\begin{equation} \label{e4.17}
x'\mik\frac{p}{1-p}\delta \ \ \ \text{ and } \ \ \ z'\mik \frac{p}{1-p}\delta
\end{equation}
that yields that
\begin{equation} \label{e4.18}
x \meg-\delta \ \ \ \ \text{ and } \ \ \ \ z\meg-\delta.
\end{equation}
By \eqref{e4.6}, we have that either
\begin{enumerate}
\item[($\mathcal{A}1$)] $x<\frac{\delta}{2}$, or
\item[($\mathcal{A}2$)] $z<\frac{\delta}{2}$.
\end{enumerate}
\begin{claim} \label{c4.3}
If $x<\frac{\delta}{2}$, then
\begin{equation} \label{e4.19}
y\mik 2\frac{p}{1-p}\delta.
\end{equation}
On the other hand, if $z<\frac{\delta}{2}$, then
\begin{equation} \label{e4.20}
w\mik  \frac{\frac{p}{1-p}+\frac{p'}{2(1-p')}}{1-\frac{p'}{2(1-p')}\delta}\, \delta.
\end{equation}
\end{claim}
\begin{proof}[Proof of Claim \emph{\ref{c4.3}}]
First assume that $x<\frac{\delta}{2}$. Then, by \eqref{e4.7}, we have
\begin{equation} \label{e4.21}
1+y\mik \frac{1+\frac{p}{1-p}\delta}{1+x'}=
1+\frac{\frac{p}{1-p}\delta-x'}{1+x'}=
1+\frac{\frac{p}{1-p}\delta+\frac{p}{1-p}x}{1-\frac{p}{1-p}x}.
\end{equation}
Since $-\frac{1}{10}<-\delta\mik x<\frac{\delta}{2}<\frac{1}{20}$ and the function
$(-1,1)\ni x\mapsto \frac{\frac{p}{1-p}\delta+\frac{p}{1-p}x}{1-\frac{p}{1-p}x}$ is increasing,
by \eqref{e4.21}, we see that \eqref{e4.19} is satisfied.

Next assume that $z<\frac{\delta}{2}$. By \eqref{e4.8}, we have
\begin{equation} \label{e4.22}
1+w \mik \frac{1+\frac{p}{1-p}\delta}{1+z'} = 1+\frac{\frac{p}{1-p}\delta-z'}{1+z'}=
1+\frac{\frac{p}{1-p}\delta+\frac{p'}{1-p'}z}{1-\frac{p'}{1-p'}z}.
\end{equation}
Thus, \eqref{e4.20} follows from \eqref{e4.22} using the fact that
$-\frac{1}{10}<-\delta\mik z< \frac{\delta}{2}<\frac{1}{20}$ and the fact that the function
$(-1,1)\ni z\mapsto \frac{\frac{p}{1-p}\delta+\frac{p'}{1-p'}z}{1-\frac{p'}{1-p'}z}$ is increasing.
\end{proof}

We proceed by considering the following cases.

\subsubsection*{Case 1: $x\meg 0$ and $z\meg 0$.} First assume that $x<\frac{\delta}{2}$.
Then, since $p\mik \frac{1}{2}$, by \eqref{e4.19} we see that $\frac{xy}{1-p}< 2\frac{p}{1-p}\delta^2$.
Moreover, since $x,z\meg 0$, we have $\frac{1-2p}{1-p}x+\frac{1-2p'}{1-p'}z(1+x)\meg 0$ and so, by~\eqref{e4.15},
we obtain that
\begin{equation} \label{e4.23}
(1+x)(1+y')> 1-\frac{p}{1-p}\delta-2\frac{p}{1-p}\delta^2\meg 1-\delta-2\frac{p}{1-p}\delta^2,
\end{equation}
that is, \eqref{e4.9} is satisfied.

On the other hand, if $z<\frac{\delta}{2}$,
then we consider the subcases ``$p\mik \frac37$" and ``$\frac37<p$"\!.
First observe that, by \eqref{e4.20} and the fact that $\delta<\frac{1}{10}$, we have that $w\mik 2\delta$.
If $0<p\mik \frac{3}{7}$ then, since $w\meg 0$, this yields that
$\frac{zw}{1-p'}<2\delta^2<\frac{1-2p}{1-p}\delta$. Using again our starting assumption that
$x,z\meg 0$ we see that $\frac{1-2p}{1-p}x(1+z)+\frac{1-2p'}{1-p'}z\meg 0$. By \eqref{e4.16} and the previous
observations, we obtain~that
\begin{equation} \label{e4.24}
(1+z)(1+w')\meg 1-\frac{p}{1-p}\delta-\frac{1-2p}{1-p}\delta=
1-\delta>1-\delta-2\frac{p}{1-p}\delta^2;
\end{equation}
in other words, in this subcase, \eqref{e4.10} is satisfied. Finally, assume that $\frac37<p\mik \frac12$.
Then observe that $\frac{p'}{2(1-p')}<\frac{2}{3}\frac{p}{1-p}$, which in turn implies,
by \eqref{e4.20} and the fact that $\delta<\frac{1}{10}$, that $w< 2\frac{p}{1-p}\delta$.
Therefore, $zw<\frac{p}{1-p}\delta^2<2\frac{p}{1-p}\delta^2$. By the previous discussion,
\eqref{e4.16} and using once again the estimate $\frac{1-2p}{1-p}x(1+z)+\frac{1-2p'}{1-p'}z\meg 0$, we conclude that
\begin{equation} \label{e4.25}
(1+z)(1+w') > 1-\delta-2\frac{p}{1-p}\delta^2
\end{equation}
and so, in this subcase, \eqref{e4.10} is satisfied.

\subsubsection*{Case 2: $x<0$ and $z\meg 0$.} Then we have
$-\frac{xy}{1-p}\meg 0$ and $\frac{1-2p}{1-p}x+\frac{1-2p'}{1-p'}z(1+x)\meg \frac{1-2p}{1-p}x$
and so, by \eqref{e4.15} and \eqref{e4.18}, we obtain that
\begin{equation} \label{e4.26}
(1+x)(1+y')\meg 1-\frac{p}{1-p}\delta-\frac{1-2p}{1-p}\delta =
1-\delta>1-\delta-2\frac{p}{1-p}\delta^2.
\end{equation}
Thus, in this case, \eqref{e4.9} is satisfied.

\subsubsection*{Case 3: $x\meg 0$ and $z<0$.} It is similar to Case 2.
Indeed, observe that $-\frac{zw}{1-p}\meg 0$ and
$\frac{1-2p}{1-p}x(1+z)+\frac{1-2p'}{1-p'}z\meg \frac{1-2p'}{1-p'}z$.
Hence, by \eqref{e4.16}, \eqref{e4.18} and the fact that $\frac{1-2p}{1-p}\meg\frac{1-2p'}{1-p'}$,
we~obtain that
\begin{equation} \label{e4.27}
(1+z)(1+w')\meg 1-\frac{p}{1-p}\delta-\frac{1-2p'}{1-p'}\delta >1-\delta-2\frac{p}{1-p}\delta^2;
\end{equation}
thus, in this case, \eqref{e4.10} is satisfied.

\subsubsection*{Case 4: $x<0$ and $z<0$.} First observe that $-\frac{xy}{1-p}>0$.
Moreover, since $z'\mik y$, by~\eqref{e4.7},
\begin{equation} \label{e4.28}
1+x'+z'+x'z'=(1+x')(1+z') \mik (1+x')(1+y) \mik 1+\frac{p}{1-p}\delta
\end{equation}
that implies that $x'+z'+x'z'\mik \frac{p}{1-p}\delta$. Noticing that $\frac{1-2p'}{p'}\mik \frac{1-2p}{p}$,
by \eqref{e4.4} and \eqref{e4.5},
\begin{align}
\label{e4.29} \frac{1-2p}{1-p}x+\frac{1-2p'}{1-p'}z(1+x) & =
-\frac{1-2p}{p}x'-\frac{1-2p'}{p'}z'+\frac{1-2p'}{p'}\frac{1-p}{p}x'z' \\
& \meg -\frac{1-2p}{p}x'-\frac{1-2p}{p}z'-\frac{1-2p}{p}x'z' \nonumber \\
& \meg -\frac{1-2p}{p}(x'+z'+x'z')\meg  -\frac{1-2p}{1-p}\delta. \nonumber
\end{align}
Hence, by \eqref{e4.15} and \eqref{e4.29}, we conclude that
\begin{equation} \label{e4.30}
(1+x)(1+y')\meg 1-\frac{p}{1-p}\delta-\frac{1-2p}{1-p}\delta=
1-\delta>1-\delta-2\frac{p}{1-p}\delta^2;
\end{equation}
thus, in this case, \eqref{e4.9} is satisfied.
\medskip

The above cases are exhaustive, and so, the proof of Sublemma \ref{s4.2} is completed.


\section{The algorithm} \label{sec5}

\numberwithin{equation}{section}

In this section we present the formal description and the basic properties of
the algorithm that is used in the proof of Theorem~\ref{main}; we shall also discuss in more
loose terms its main features. (Again, we recall that for every family
$\mathcal{A}\subseteq \{0,1\}^n$ ($n\meg 2$) by $\mathcal{A}_0,\mathcal{A}_1\subseteq \{0,1\}^{n-1}$
we denote the sections of $\Acal$ defined in \eqref{e2.1}.)

For the analysis of the algorithm, we will need the following elementary, but crucial,
fact that originates in the work of Frankl and R\"{o}dl \cite{FR87}.
\begin{fact} \label{f5.1}
Let $n\meg 2$ be an integer, let $\Fcal,\Gcal\subseteq \{0,1\}^n$ and let $a,b\in [n]$ with $a\mik b$.
Assume that $(\Fcal,\Gcal)\in\mathrm{Forbid}(n,[a,b])$. Then we have
\begin{align}
\label{e5.1}  (\Fcal_1,\Gcal_1) \in  \mathrm{Forbid}&(n-1,[a-1,b-1]), \\
\label{e5.2}  (\Fcal_0,\Gcal_0\cup\Gcal_1)\in \mathrm{Forbid}(n-1,[a,b]), & \ \
(\Fcal_0\cup\Fcal_1,\Gcal_0)\in \mathrm{Forbid}(n-1,[a,b]), \\
\label{e5.3} (\Fcal_1,\Gcal_0\cap\Gcal_1)\in \mathrm{Forbid}(n-1,[a-1,b]), & \ \
(\Fcal_0\cap\Fcal_1,\Gcal_1)\in \mathrm{Forbid}(n-1,[a-1,b]).
\end{align}
\end{fact}

\begin{algorithm}
\hrule
\vspace{0.11cm}

\textbf{Algorithm.}

\vspace{0.1cm}
\hrule
\BlankLine

\textbf{Input:} $n,p,p',\ell,0<\delta<\frac{1}{10},
(\Fcal_{\mathrm{init}}, \Gcal_{\mathrm{init}})\in\mathrm{Forbid}(n,\{\ell\})$
with $\Fcal_{\mathrm{init}}, \Gcal_{\mathrm{init}}\neq\emptyset$ \\
\textbf{Output:} $a^*,b^*,m^*, \Fcal^*, \Gcal^*$
\BlankLine

\textbf{Initialize:} $S_{d_1}\leftarrow 0$, $S_{d_2}\leftarrow 0$, $S_{w}\leftarrow 0$,
$a\leftarrow \ell$, $b\leftarrow \ell$, \\
\hspace{1.71cm} $m\leftarrow n$, $\Fcal\leftarrow\Fcal_{\mathrm{init}}$, $\Gcal\leftarrow\Gcal_{\mathrm{init}}$.
\BlankLine

\textbf{Step 1:} If $a=0$, then $a^*\leftarrow a$, $b^*\leftarrow b$, $m^*\leftarrow m$,
$\Fcal^*\leftarrow\Fcal$, $\Gcal^*\leftarrow\Gcal$ and terminate; \\
\hspace{1.3cm} else go to \textbf{Step 2}.\\
\textbf{Step 2:} If $b=m$, then $a^*\leftarrow a$, $b^*\leftarrow b$, $m^*\leftarrow m$,
$\Fcal^*\!\leftarrow\Fcal$, $\Gcal^*\leftarrow\Gcal$ and terminate; \\
\hspace{1.3cm} else go to \textbf{Step 3}.\\
\textbf{Step 3:} If $\mu_p(\Fcal_1)\, \mu_{p'}(\Gcal_1)> (1+\delta)\,\mu_p(\Fcal)\, \mu_{p'}(\Gcal)$,
then $S_{d_1}\leftarrow S_{d_1}+1$, \\
\hspace{1.6cm} $a\leftarrow a-1$, $b\leftarrow b-1$, $m\leftarrow m-1$,
$\Fcal\leftarrow\Fcal_1$, $\Gcal\leftarrow\Gcal_1$ and go to \textbf{Step 1}; \\
\hspace{1.3cm} else go to \textbf{Step 4}.  \\
\textbf{Step 4:} If $\mu_p(\Fcal_0)\, \mu_{p'}(\Gcal_0\cup\Gcal_1)>(1+\frac{p}{1-p}\delta)\,
\mu_p(\Fcal)\, \mu_{p'}(\Gcal)$, then $S_{d_2}\leftarrow S_{d_2}+1$, \\
\hspace{1.6cm} $m\leftarrow m-1$, $\Fcal\leftarrow\Fcal_0$, $\Gcal\leftarrow\Gcal_0\cup\Gcal_1$
and go to \textbf{Step 1}; \\
\hspace{1.3cm} else go to \textbf{Step 5}. \\
\textbf{Step 5:} If $\mu_{p}(\Fcal_0\cup\Fcal_1)\, \mu_{p'}(\Gcal_0)>(1+\frac{p}{1-p}\delta)\,
\mu_p(\Fcal)\, \mu_{p'}(\Gcal)$, then $S_{d_2}\leftarrow S_{d_2}+1$, \\
\hspace{1.6cm} $m\leftarrow m-1$, $\Fcal\leftarrow\Fcal_0\cup \Fcal_1$, $\Gcal\leftarrow\Gcal_0$
and go to \textbf{Step 1}; \\
\hspace{1.3cm} else go to \textbf{Step 6}. \\
\textbf{Step 6:} If $\mu_p(\Fcal_1)\, \mu_{p'}(\Gcal_0\cap\Gcal_1)> (1-\delta-2\frac{p}{1-p}\delta^2)\,
\mu_p(\Fcal)\,\mu_{p'}(\Gcal)$, \\
\hspace{1.6cm} then $S_{w}\leftarrow S_{w}+1$, $a\leftarrow a-1$, $m\leftarrow m-1$, $\Fcal\leftarrow\Fcal_1$,
$\Gcal\leftarrow\Gcal_0\cap\Gcal_1$ \\
\hspace{1.6cm} and go to \textbf{Step 1}; \\
\hspace{1.3cm} else go to \textbf{Step 7}. \\
\textbf{Step 7:} $S_{w}\leftarrow S_{w}+1$, $a\leftarrow a-1$, $m\leftarrow m-1$,
$\Fcal\leftarrow\Fcal_0\cap\Fcal_1$, $\Gcal\leftarrow\Gcal_1$ \\
\hspace{1.6cm} and go to \textbf{Step 1}.
\vspace{0.1cm}
\hrule
\end{algorithm}

The algorithm takes as an input
\begin{enumerate}
\item[(I1)] an integer $n\meg 2$, two reals $0<p\mik p'\mik \frac12$ and a positive integer $\ell< pn$,
\item[(I2)] a real $0<\delta<\frac{1}{10}$, and
\item[(I3)] two nonempty families $\Fcal_{\mathrm{init}}, \Gcal_{\mathrm{init}}\subseteq \{0,1\}^n$ whose
cross intersections forbid~$\ell$,
\end{enumerate}
and outputs
\begin{enumerate}
\item[(O1)] three nonnegative integers $a^*\mik b^*\mik m^*$ with $m^*\meg 1$, and
\item[(O2)] two families $\Fcal^*, \Gcal^*\subseteq \{0,1\}^{m^*}$ such that
\begin{equation} \label{e5.4}
(\Fcal^*,\Gcal^*)\in\Forb{m^*}{[0,b^*]} \ \  \text{ or } \ \
(\Fcal^*,\Gcal^*)\in \Forb{m^*}{[a^*,m^*]}.
\end{equation}
\end{enumerate}
It also uses six counters $S_{d_1},S_{d_2},S_w,a,b,m$ that serve different purposes.
The first three counters, $S_{d_1},S_{d_2}$ and $S_w$, give us the total number of iterations
and they are used for bookkeeping the operations performed by the algorithm (we shall comment
on these operations in due course). The counters $a,b$ encode the interval that is forbidden
for the families $\Fcal$ and $\Gcal$. Finally, the counter $m$ keeps track of the dimension of
$\Fcal$ and $\Gcal$; in particular, $m$ starts from $n$ and drops by one at each iteration. Thus, we have
\begin{equation}\label{e5.5}
S_{d_1}+S_{d_2}+S_w=n-m^*.
\end{equation}
Moreover, by Fact \ref{f5.1}, at each iteration of ``type" $S_{d_1}$ or $S_{w}$
(that is, at each iteration where one of the counters $S_{d_1}$ or $S_{w}$ is increased by one)
the lower bound $a$ of the forbidden interval $[a,b]$ is reduced by one; since the algorithm starts
with $a=\ell$, we obtain that
\begin{equation} \label{e5.6}
S_{d_1}+S_w\mik \ell.
\end{equation}

Next observe that $S_{d_1}$ and $S_{d_2}$ count the number of iterations where we have ``density increment"\!.
Note, however, that this increment is not uniform: at each iteration of ``type" $S_{d_1}$
the product of the measures is increased by a factor $(1+\delta)$, while at each iteration of ``type" $S_{d_2}$
the product of the measures is increased by a factor $(1+\frac{p}{1-p}\delta)$. On the other hand, if at a certain
iteration the algorithm reaches \textbf{Step 6} and then moves to \textbf{Step 1}, then the product of the measures
of the new families is comparable to the product of the measures of the previous families by a factor
$(1-\delta-2\frac{p}{1-p}\delta^2)$; in other words, the product of the measures may possibly drop,
but not significantly. Finally,  if at a certain iteration the algorithm reaches \textbf{Step 7},
then the widening lemma (Lemma \ref{l4.1}) ensures that the product of the measures
of the new families is also at least $(1-\delta-2\frac{p}{1-p}\delta^2)$ times
the product of the measures of the previous families.

Summing up the previous observations, we arrive at the following basic estimate
\begin{equation} \label{e5.7}
\mu_p(\Fcal^*)\,\mu_{p'}(\Gcal^*) > (1+\delta)^{S_{d_1}} \Big(1+\frac{p}{1-p}\delta\Big)^{S_{d_2}}\,
\Big(1-\delta-2\frac{p}{1-p}\delta^2\Big)^{S_w}\, \mu_p(\Fcal_{\mathrm{init}})\, \mu_{p'}(\Gcal_{\mathrm{init}})
\end{equation}
that will be used in the analysis of the algorithm in the next section.


\section{Proof of Theorem \ref*{main}} \label{sec6}

\numberwithin{equation}{section}

Let $n,p,p',\ell,\Fcal,\Gcal$ be as in the statement of the theorem. Clearly, we may assume that
$\Fcal$ and $\Gcal$ are nonempty. Notice that \eqref{e1.main} is straightforward if $\ell=0$ or $\ell=pn$.
Thus, we may also assume that $\ell$ is a positive integer with $\ell<pn$ and, consequently, $n\meg 2$;
moreover, setting
\begin{equation} \label{e6.1}
\delta\coloneqq \min\Big\{ \frac{\ell}{58pn}, \frac{pn-\ell}{51 pn}\Big\},
\end{equation}
we have that $0<\delta<\frac{1}{10}$. We will actually show the slightly stronger estimate
\begin{equation} \label{e6.2}
\mu_p(\Fcal)\, \mu_{p'}(\Gcal)\mik 2 \exp(-pn \delta^2).
\end{equation}
Assume, towards a contradiction, that this is not the case, that is,
\begin{equation} \label{e6.3}
\mu_p(\Fcal)\,\mu_{p'}(\Gcal)>2\exp(-pn\delta^2).
\end{equation}
We run the algorithm described in Section \ref{sec5} for $n,p,p',\ell,\delta$ and the families
$\Fcal,\Gcal$. Let $a^*,b^*,m^*,\Fcal^*,\Gcal^*$ denote the output of the algorithm.
By \eqref{e5.4}, we see that either
\begin{enumerate}
\item[($\mathcal{A}1$)] $(\Fcal^*,\Gcal^*)\in\mathrm{Forbid}(m^*,[0,b^*])$, or
\item[($\mathcal{A}2$)] $(\Fcal^*,\Gcal^*)\in\mathrm{Forbid}(m^*,[a^*,m^*])$.
\end{enumerate}
The contradiction will be derived by showing that none of these cases can occur.

To this end we first observe that, by \eqref{e5.7} and \eqref{e6.3}, we have
\begin{equation} \label{e6.4}
1 >(1+\delta)^{S_{d_1}}\, \Big(1+\frac{p}{1-p}\delta\Big)^{S_{d_2}}\,
\Big(1-\delta-2\frac{p}{1-p}\delta^2\Big)^{S_w}\, \exp(-pn\delta^2),
\end{equation}
where $S_{d_1},S_{d_2},S_w$ are the counters used in the algorithm.
We will need the following estimates for $S_{d_1},S_{d_2}$ and $S_w$.
\begin{lem} \label{l6.1}
We have
\begin{align}
\label{e6.5} & \ \ \ \, S_{d_1}-S_w < 5pn\delta, \\
\label{e6.6} & S_{d_2} -\frac{1-p}{p}S_w  < 3n \delta.
\end{align}
\end{lem}
In the proof of Lemma \ref{l6.1}, as well as in the rest of this section, we will repeatedly use
the following elementary observation, which we isolate for the convenience of the reader.
\begin{observation} \label{o6.2}
The following hold.
\begin{enumerate}
\item[(i)] \label{i} We have $\frac{1}{1-x}=1+\frac{x}{1-x}$ for every $x\neq 1$.
\item[(ii)] \label{ii} We have $x-\frac{x^2}{2}\mik\ln(1+x)\mik x$ for every $x\meg 0$.
\end{enumerate}
In particular, for every $0\mik x <1$ we have
\begin{equation} \label{en.6.1}
\frac{x}{1-x} -\frac{x^2}{2(1-x)^2} \mik \ln\Big(\frac{1}{1-x}\Big) \mik \frac{x}{1-x}.
\end{equation}
\end{observation}
We are ready to proceed to the proof of Lemma \ref{l6.1}.
\begin{proof}[Proof of Lemma \emph{\ref{l6.1}}]
We start with the proof of \eqref{e6.5}. Notice first that, by \eqref{e6.4},
\begin{equation} \label{e6.7}
1 > (1+\delta)^{S_{d_1}} \, \Big(1-\delta-2\frac{p}{1-p}\delta^2\Big)^{S_w}\exp(-pn\delta^2)
\end{equation}
or, equivalently,
\begin{equation} \label{e6.8}
1 >(1+\delta)^{S_{d_1}-S_w}
\Big((1+\delta) \big(1-\delta-2\frac{p}{1-p}\delta^2\big)\Big)^{S_w} \exp(-pn\delta^2).
\end{equation}
Since $(1+\delta)(1-\delta-2\frac{p}{1-p}\delta^2)=1-\frac{1+p}{1-p}\delta^2-2\frac{p}{1-p}\delta^3$,
after taking logarithms and rearranging we find that
\begin{equation} \label{e6.9}
S_{d_1}-S_w < \frac{1}{\ln(1+\delta)}
\left(S_w\, \ln\Big(\frac{1}{1-\frac{1+p}{1-p}\delta^2-2\frac{p}{1-p}\delta^3}\Big)+pn\delta^2\right)
\end{equation}
that implies, by Observation \ref{o6.2}, that
\begin{equation} \label{e6.10}
S_{d_1}-S_w < \frac{1}{\delta(1-\frac{\delta}{2})} \left( S_w \, \frac{\frac{1+p}{1-p}\delta^2+2\frac{p}{1-p}\delta^3}{1-\frac{1+p}{1-p}\delta^2-2\frac{p}{1-p}\delta^3}
+pn\delta^2\right).
\end{equation}
By \eqref{e5.6}, the fact that $\ell< pn$ and \eqref{e6.10}, we conclude that
\begin{equation} \label{e6.11}
S_{d_1}-S_w<pn\delta \left(\frac{\frac{1+p}{1-p}+2\frac{p}{1-p}\delta}{\left(1-\frac{\delta}{2}\right)
\left(1-\frac{1+p}{1-p}\delta^2-2\frac{p}{1-p}\delta^3\right)}+\frac{1}{\left(1-\frac{\delta}{2}\right)}\right).
\end{equation}
The desired estimate \eqref{e6.5} follows from \eqref{e6.11} and the fact that $0<\delta<\frac{1}{10}$ and $p\mik \frac12$.

We proceed to show that inequality \eqref{e6.6} is also satisfied.
As before, we first observe that \eqref{e6.4} yields that
\begin{equation} \label{e6.12}
1 > \Big(1+\frac{p}{1-p}\delta\Big)^{S_{d_2}-S_w}
\left(\big(1+\frac{p}{1-p}\delta\big)\big(1-\delta-2\frac{p}{1-p}\delta^2\big)\right)^{S_w}\, \exp(-pn\delta^2).
\end{equation}
On the other hand, since $0<\delta<\frac{1}{10}$, we have
\begin{equation} \label{e6.13}
\Big(1+\frac{p}{1-p}\delta\Big)\, \Big(1-\delta-2\frac{p}{1-p}\delta^2\Big)\meg
1-\frac{1-2p}{1-p}\delta-\frac{16}{5}\frac{p}{1-p}\delta^2
\end{equation}
that combined with \eqref{e6.12} yields that
\begin{equation} \label{e6.14}
1 > \Big(1+\frac{p}{1-p}\delta\Big)^{S_{d_2}-S_w}
\Big(1-\frac{1-2p}{1-p}\delta-\frac{16}{5}\frac{p}{1-p}\delta^2\Big)^{S_w}\, \exp(-pn\delta^2).
\end{equation}
Now after taking logarithms and rearranging, we have
\begin{equation} \label{e6.15}
S_{d_2}-S_w < \frac{1}{\ln\left(1+\frac{p}{1-p}\delta\right)} \,
\left(S_w\, \ln\left(\frac{1}{1-\frac{1-2p}{1-p}\delta-\frac{16}{5}\frac{p}{1-p}\delta^2}\right)+pn\delta^2\right);
\end{equation}
by Observation \ref{o6.2}, this yields that
\begin{equation} \label{e6.16}
S_{d_2}-S_w < \frac{1-p}{p\delta\left(1-\frac{p}{2(1-p)}\delta\right)} \,
\left(S_w\, \frac{\frac{1-2p}{1-p}\delta+\frac{16}{5}
\frac{p}{1-p}\delta^2}{1-\frac{1-2p}{1-p}\delta-\frac{16}{5}\frac{p}{1-p}\delta^2} + pn\delta^2\right)
\end{equation}
that can be further simplified to
\begin{equation} \label{e6.17}
S_{d_2}-S_w <  S_w\, \frac{\frac{1-2p}{p}+\frac{16}{5}\delta}{\left(1-\frac{p}{2(1-p)}\delta\right)
\left(1-\frac{1-2p}{1-p}\delta-\frac{16}{5}\frac{p}{1-p}\delta^2\right)}+
\frac{(1-p)n\delta}{1-\frac{p}{2(1-p)}\delta}.
\end{equation}
On the other hand, since $p\mik \frac12$, for every $0<\delta<\frac{1}{10}$ we have
\begin{equation} \label{e6.18}
\left(1-\frac{p}{2(1-p)}\delta\right)
\left(1-\frac{1-2p}{1-p}\delta-\frac{16}{5}\frac{p}{1-p}\delta^2\right)\meg 1-\delta;
\end{equation}
indeed, after noticing that
\begin{equation} \label{en.6.2}
1-\frac{1-2p}{1-p}\delta-\frac{16}{5}\frac{p}{1-p}\delta^2 =
1-\delta+ \frac{p}{1-p}\delta -\frac{16}{5}\, \frac{p}{1-p}\delta^2,
\end{equation}
the desired estimate \eqref{e6.18} follows from the elementary inequality
\begin{equation} \label{en.6.3}
\frac{p}{1-p}\delta + \frac{16}{5}\, \frac{p^2}{2(1-p)^2}\delta^3
- \frac{p}{2(1-p)}\delta(1-\delta) - \frac{p^2}{2(1-p)^2}\delta^2
- \frac{16}{5}\, \frac{p}{1-p}\delta^2 \meg 0.
\end{equation}
By \eqref{e6.17} and \eqref{e6.18}, we obtain that
\begin{equation} \label{e6.19}
S_{d_2}-S_w < S_w\, \left(\frac{1-2p}{p}+\frac{16}{5}\delta\right)
\left(1+\frac{10}{9}\delta\right)+(1-p)n\delta\left(1+\delta\right).
\end{equation}
We then expand \eqref{e6.19} to
\begin{equation} \label{e6.20}
S_{d_2}-S_w < S_w\, \left(\frac{1-2p}{p}+ \frac{16}{5}\delta+ \frac{1-2p}{p}\frac{10}{9}\delta +\frac{10}{9}\frac{16}{5}\delta^2\right)+(1-p)n\delta\left(1+\delta\right).
\end{equation}
By \eqref{e5.6}, we see that $S_w\mik \ell< pn$, and so \eqref{e6.20} yields that
\begin{equation} \label{e6.21}
S_{d_2}-\frac{1-p}{p}S_w < n\delta \left(p\frac{16}{5}+(1-2p)\frac{10}{9}+p\frac{10}{9}\frac{16}{5}\delta+(1-p)\left(1+\delta\right)\right).
\end{equation}
Inequality \eqref{e6.6} follows from \eqref{e6.21} and the fact that
$0<\delta<\frac{1}{10}$ and $0<p\mik\frac{1}{2}$.
\end{proof}
After these preliminary steps, we are ready to consider cases.

\subsection*{Case 1: $(\Fcal^*,\Gcal^*)\in\mathrm{Forbid}(m^*,[0,b^*])$}

Note that, in this case, the cardinality of the final forbidden interval $[0,b^*]$ is $b^*+1$;
on the other hand, the forbidden interval for the initial families $\Fcal,\Gcal$ was a singleton.
By Fact \ref{f5.1}, the cardinality of the forbidden interval increases by $1$
if and only if the algorithm executes an iteration of ``type'' $S_w$. Thus,
\begin{equation} \label{e6.22}
b^*=S_w.
\end{equation}
Next observe that the initial value of the lower bound $a$ of the forbidden interval is equal to $\ell$,
and it is equal to $0$ when the algorithm terminates. Using Fact \ref{f5.1} again, we see that
the counter $a$ decreases by $1$ if and only if an iteration of ``type" $S_{d_1}$ or an iteration
of ``type" $S_w$ is executed. Therefore, we also have that
\begin{equation} \label{e6.23}
S_{d_1}+S_w=\ell.
\end{equation}
By \eqref{e6.5} and \eqref{e6.23}, we obtain that
\begin{equation} \label{e6.24}
S_w\meg \frac{\ell}{2}-\frac52 pn\delta.
\end{equation}

On the other hand, since $\left(\Fcal^*,\Gcal^*\right)\in\mathrm{Forbid}(m^*,[0,b^*])$ and $b^*=S_w$,
by part (i) of Lemma \ref{l3.3}, we have
\begin{equation} \label{e6.25}
\mu_p(\Fcal^*)\,\mu_{p'}(\Gcal^*)\mik \exp\left(-\frac{S_w^2}{24pn}\right).
\end{equation}
Combining \eqref{e5.7}, \eqref{e6.3}, \eqref{e6.23} and \eqref{e6.25}, we obtain in particular that
\begin{equation} \label{e6.26}
\exp\left(-\frac{S_w^2}{24pn}\right) > (1+\delta)^{\ell-S_w} \left(1-\delta-2\frac{p}{1-p}\delta^2\right)^{S_w}\,
\exp\left(-pn\delta^2\right)
\end{equation}
that implies, after taking logarithms and using Observation \ref{o6.2}, that
\begin{equation} \label{e6.27}
(S_w-\ell)\, \Big(\delta-\frac{\delta^2}{2}\Big) +
S_w\, \left(\frac{\delta+2\frac{p}{1-p}\delta^2}{1-\delta-2\frac{p}{1-p}\delta^2}\right)+pn\delta^2
> \frac{S_w^2}{24pn}.
\end{equation}
Moreover, since $0<\delta<\frac{1}{10}$ and $0<p\mik \frac12$, we have
\begin{equation} \label{e6.28}
\delta+\frac{11}{3}\delta^2 \meg \frac{\delta+2\frac{p}{1-p}\delta^2}{1-\delta-2\frac{p}{1-p}\delta^2},
\end{equation}
which, combined with \eqref{e6.27}, implies that
\begin{equation} \label{e6.29}
24pn(2S_w-\ell)\delta - 12pn(2S_w-\ell)\delta^2 + 100 pnS_w\delta^2 +
24 p^2n^2\delta^2> S_w^2.
\end{equation}
We now consider the following subcases.

\subsubsection*{Subcase 1.1: $S_w\mik \frac23\ell$} In this subcase, by \eqref{e6.24} and the choice of
$\delta$ in \eqref{e6.1}, we have
\begin{equation} \label{e6.30}
-\frac{pn}{2}\mik 2S_w-\ell \mik \frac{S_w}{2} \ \ \ \text{ and } \ \ \
\delta \mik\frac{2}{53}\, \frac{S_w}{pn};
\end{equation}
indeed, by \eqref{e6.24} and \eqref{e6.1}, $S_w\meg \frac{\ell}{2}-\frac52 pn\delta\meg
\ell \big(\frac12- \frac{5}{116}\big) \meg 58pn\delta \big(\frac12- \frac{5}{116}\big)\meg
\frac{53}{2}pn\delta$. Hence, by \eqref{e6.29}, \eqref{e6.30} and the fact that $\delta<\frac{1}{10}$, we obtain that
\begin{equation} \label{e6.31}
S_w^2 < 12pnS_w\delta +30p^2n^2\delta^2 + 100pnS_w\delta^2
\mik \Big( \frac{24}{53}+ \frac{30\cdot 4}{53^2}+ \frac{200}{53}\cdot\frac{1}{10} \Big) S_w^2 < S_w^2,
\end{equation}
which is clearly a contradiction.

\subsubsection*{Subcase 1.2: $S_w>\frac23\ell$} By \eqref{e6.23} and \eqref{e6.1}, we have
\begin{equation} \label{e6.32}
0\mik \frac{S_w}{2} \mik 2S_w-\ell \mik S_w \ \ \ \text{ and } \ \ \
\delta \mik\frac{3}{116}\, \frac{S_w}{pn}.
\end{equation}
Therefore, by \eqref{e6.29} and \eqref{e6.32}, we get that
\begin{equation} \label{e6.33}
S_w^2 < 24pnS_w\delta +100pn S_w\delta^2 + 24 p^2n^2\delta^2
< \Big( \frac{24\cdot 3}{116}+ \frac{100\cdot 3}{116}\cdot \frac{1}{10}+ \frac{24\cdot 9}{116^2} \Big) S_w^2 < S_w^2,
\end{equation}
which leads, again, to a contradiction.

\subsection*{Case 2: $(\Fcal^*,\Gcal^*)\in\mathrm{Forbid}(m^*,[a^*,m^*])$}

The proof in this case is slightly more involved.
We start by observing that the initial value of the upper bound $b$ of the forbidden interval is equal to $\ell$,
and it is equal to $m^*$ when the algorithm terminates. Moreover, by Fact \ref{f5.1}, the counter $b$ decreases
by $1$ if and only if an iteration of ``type" $S_{d_1}$ is executed.~Thus,
\begin{equation} \label{e6.34}
m^*=\ell-S_{d_1}\mik \ell.
\end{equation}
Combining \eqref{e5.5} and \eqref{e6.34}, we obtain that
\begin{equation} \label{e6.35}
S_{d_2}+S_w=n-\ell.
\end{equation}
This identity together with \eqref{e6.6} yields that $S_w\meg pn-3pn\delta-p\ell$ that we rewrite as
\begin{equation} \label{e6.36}
S_w\meg (1-p)\ell+\left(\left(pn-\ell\right)-3pn\delta\right).
\end{equation}
Moreover, as we have already noted in the previous case, the cardinality of the forbidden interval increases by $1$
if and only if the algorithm executes an iteration of ``type'' $S_w$. Therefore, we also have that
\begin{equation} \label{e6.37}
m^*-a^*=S_w.
\end{equation}

Next, we introduce the quantity
\begin{equation} \label{e6.38}
\alpha\coloneqq S_w-(1-p)\ell,
\end{equation}
and we observe that, by \eqref{e6.36}, we have the lower bound
\begin{equation} \label{e6.39}
\alpha \meg \left(pn-\ell\right)-3pn\delta;
\end{equation}
notice that $\alpha>0$ by the choice of $\delta$ in \eqref{e6.1}.
Also notice that, by \eqref{e6.37} and \eqref{e6.38},
\begin{equation} \label{e6.40}
a^*+(\ell-m^*)=p\ell-\alpha.
\end{equation}
On the other hand, setting
\begin{align}
\label{e6.41} & \widehat{\Fcal}\coloneqq \{A\subseteq [\ell]:
A\cap [m^*]\in \Fcal^*\}\subseteq \{0,1\}^\ell, \\
\label{e6.42} & \widehat{\Gcal}\coloneqq \{B\subseteq [\ell]:
B\cap [m^*]\in \Gcal^*\}\subseteq \{0,1\}^\ell,
\end{align}
by \eqref{e6.40} and the fact that $(\Fcal^*,\Gcal^*)\in\mathrm{Forbid}(m^*,[a^*,m^*])$, we see that
\begin{enumerate}
\item[(i)] $\mu_p(\Fcal^*)=\mu_p(\widehat{\Fcal})$
and $\mu_{p'}(\Gcal^*)=\mu_{p'}(\widehat{\Gcal})$, and
\item[(ii)] $(\widehat{\Fcal},\widehat{\Gcal})\in \mathrm{Forbid}(\ell,[p\ell-\alpha,\ell])$.
\end{enumerate}
Hence, by part (ii) of Lemma \ref{l3.3}, we obtain that
\begin{equation} \label{e6.43}
\mu_p(\Fcal^*)\, \mu_{p'}(\Gcal^*)\mik 2\exp\left(-\frac{\alpha^2}{24p\ell}\right).
\end{equation}
In particular, by \eqref{e5.7}, \eqref{e6.3} and \eqref{e6.43}, we have that
\begin{equation} \label{e6.44}
\exp\left(-\frac{\alpha^2}{24p\ell}\right) > \left(1+\frac{p}{1-p}\delta\right)^{S_{d_2}}
\left(1-\delta-2\frac{p}{1-p}\delta^2\right)^{S_w}\, \exp(-pn\delta^2).
\end{equation}
As in the previous case, we will show that \eqref{e6.44} leads to a contradiction.

To this end it is enough to show that, by the choice of $\delta$ in \eqref{e6.1}, we have
\begin{equation} \label{e6.45}
\exp\left(-\frac{\alpha^2}{24p\ell}\right) \mik \left(1+\frac{p}{1-p}\delta\right)^{S_{d_2}}
\left(1-\delta-2\frac{p}{1-p}\delta^2\right)^{S_w}\, \exp(-pn\delta^2).
\end{equation}
After taking logarithms and using Observation \ref{o6.2}, it is enough to show that
\begin{equation} \label{e6.46}
pn\delta^2+S_w\, \frac{\delta+2\frac{p}{1-p}\delta^2}{1-\delta-2\frac{p}{1-p}\delta^2} \mik S_{d_2}\left(\frac{p}{1-p}\delta-\frac{p^2}{2(1-p)^2}\delta^2\right)+\frac{\alpha^2}{24p\ell}.
\end{equation}
Since $0<\delta<\frac{1}{10}$ and $p\mik \frac12$, we have
\begin{align}
\label{e6.47} & \ \ \ \ \ \ \, \frac{1}{1-\delta-2\frac{p}{1-p}\delta^2} < \left(1+\frac32\delta\right), \\
\label{e6.48} & \left(1+2\frac{p}{1-p}\delta\right) \left(1+\frac32\delta\right)\mik (1+4\delta).
\end{align}
Thus, by \eqref{e6.46}--\eqref{e6.48}, it is enough to show that
\begin{equation} \label{e6.49}
pn\delta^2+S_w\delta(1+4\delta)\mik S_{d_2}\, \left(\frac{p}{1-p}\delta-\frac{p^2}{2(1-p)^2}\delta^2\right)
+\frac{\alpha^2}{24p\ell},
\end{equation}
which is equivalent to saying, after rearranging, that
\begin{equation} \label{e6.50}
\left(pn+4S_w+\frac{p^2}{2(1-p)^2}\, S_{d_2}\right)\delta^2+\left(S_w-\frac{p}{1-p}\,S_{d_2}\right)\delta
\mik \frac{\alpha^2}{24p\ell}.
\end{equation}
By \eqref{e5.6}, we have $S_w\mik\ell< pn$ and, clearly, $S_{d_2}\mik n$.
Hence, by \eqref{e6.50} and the fact that $0<p\mik \frac12$, it is enough to show that
\begin{equation} \label{e6.51}
6pn\delta^2 +\left(S_w-\frac{p}{1-p}\,S_{d_2}\right)\delta \mik \frac{\alpha^2}{24p\ell}.
\end{equation}
Observe that
\begin{align}
\label{e6.52} S_w-\frac{p}{1-p}S_{d_2} & \stackrel{\eqref{e6.35}}{=}
S_w-\frac{p}{1-p}\left(n-\ell-S_w\right) =\frac{1}{1-p}(S_w-pn+p\ell) \\
& \stackrel{\eqref{e6.38}}{=} \frac{1}{1-p}\big(\alpha+(1-p)\ell-pn+p\ell\big) =
\frac{1}{1-p}\big(\alpha-(pn-\ell)\big). \nonumber
\end{align}
In order to verify \eqref{e6.51}, we consider the following subcases.

\subsubsection*{Subcase 2.1: $\alpha\mik pn-\ell$} By \eqref{e6.52},
we have $S_w-\frac{p}{1-p}S_{d_2}\mik 0$, and so it is enough to show that
\begin{equation} \label{e6.53}
144p^2n^2\delta^2 \mik \alpha^2.
\end{equation}
Since $0<\delta<\frac{pn-\ell}{51pn}$, by the choice of $\delta$ in \eqref{e6.1},
the estimate \eqref{e6.53}---and, consequently, \eqref{e6.51}---follows from \eqref{e6.39}.

\subsubsection*{Subcase 2.2: $\alpha> pn-\ell$} In this subcase, by \eqref{e6.52}, we have $S_w-\frac{p}{1-p}S_{d_2}>0$.
Hence, using again the fact that $0<\delta<\frac{pn-\ell}{51pn}$, it is enough to show that
\begin{equation} \label{e6.54}
\frac{144}{51^2}(pn-\ell)^2 + \frac{24}{51}\left(S_w-\frac{p}{1-p}S_{d_2}\right)(pn-\ell)\mik \alpha^2,
\end{equation}
which is equivalent to saying, by \eqref{e6.52}, that
\begin{equation} \label{e6.55}
\left(\frac{144}{51^2}-\frac{24}{51(1-p)}\right) (pn-\ell)^2 + \frac{24}{51(1-p)}\alpha (pn-\ell)\mik \alpha^2.
\end{equation}
Since $0<p\mik \frac12$, it is enough to show that
\begin{equation} \label{e6.56}
\frac{48}{51} (pn-\ell)\mik \alpha
\end{equation}
that follows from our starting assumption that $\alpha>pn-\ell$.
\medskip

Summing up, we conclude that \eqref{e6.51} is satisfied, and as we have already indicated, this
contradicts \eqref{e6.44}. This completes the proof that Case 2 cannot occur, and so the entire
proof of Theorem \ref{main} is completed.


\section{Extensions of the main estimate} \label{sec7}

\numberwithin{equation}{section}

We start with the following proposition, which is the analogue of Theorem \ref{main} for families
of sets contained in layers of the cube.
\begin{prop} \label{p7.1}
Let $\ell\mik k\mik m\mik n$ be positive integers, and let $\Fcal\subseteq\binom{[n]}{k}$
and\, $\Gcal\subseteq\binom{[n]}{m}$ with $(\Fcal,\Gcal)\in\mathrm{Forbid}(n,\{\ell\})$.
\begin{enumerate}
\item[(i)] If $k\mik \frac{n}{2}$ and $m\mik n-k$, then, setting
$t\coloneqq \min\{\ell,k-\ell\}$, we have
\begin{equation} \label{e7.1}
\frac{|\Fcal|}{\binom{n}{k}}\cdot \frac{|\Gcal|}{\binom{n}{m}} \mik 50\,
\sqrt{\frac{k(n-k)m(n-m)}{n^2}}\, \exp\Big(-\frac{t^2}{58^2\,k}\Big).
\end{equation}
\item[(ii)] If $k\mik \frac{n}{2}\mik n-k< m\mik n-k+\ell$, then, setting
$\bar{t}\coloneqq \min\{k-\ell,n-m-(k-\ell)\}$,
\begin{equation} \label{e7.2}
\frac{|\Fcal|}{\binom{n}{k}}\cdot \frac{|\Gcal|}{\binom{n}{m}} \mik 50\,
\sqrt{\frac{k(n-k)m(n-m)}{n^2}}\, \exp\Big(-\frac{\bar{t}^{\,2}}{58^2\, (n-m)}\Big).
\end{equation}
\end{enumerate}
\end{prop}
\begin{proof}
We start with the proof of part (i). Assume, first, that $m\mik \frac{n}{2}$. Then,
\begin{equation} \label{e7.3}
\frac{|\Fcal|}{\binom{n}{k}}\cdot \frac{|\Gcal|}{\binom{n}{m}}
\stackrel{\eqref{e2.5}}{\mik} 25\, \sqrt{\frac{k(n-k)m(n-m)}{n^2}}\,
\mu_{\frac{k}{n}}(\Fcal)\, \mu_{\frac{m}{n}}(\Gcal).
\end{equation}
Thus, in this case, \eqref{e7.1} follows from \eqref{e7.3} and \eqref{e1.main} applied
for ``$p=\frac{k}{n}$" and ``$p'=\frac{m}{n}$". Next, assume that $\frac{n}{2}\mik m\mik n-k$
and set $\overline{\Gcal}\coloneqq \{[n]\setminus G: G\in\Gcal\} \subseteq \binom{[n]}{n-m}$.
Notice that $(\Fcal,\overline{\Gcal})\in\mathrm{Forbid}(n,\{k-\ell\})$ and, moreover,
$k-\ell \mik k \mik n-m \mik\frac{n}{2}$. Therefore, applying the estimate obtained in the
first part of the proof to the pair $(\Fcal,\overline{\Gcal})$ and invoking the choice of $t$,
we obtain that
\begin{equation} \label{e7.4}
\frac{|\Fcal|}{\binom{n}{k}}\cdot \frac{|\Gcal|}{\binom{n}{m}} =
\frac{|\Fcal|}{\binom{n}{k}}\cdot \frac{|\overline{\Gcal}|}{\binom{n}{n-m}}
\mik 50\, \sqrt{\frac{k(n-k)m(n-m)}{n^2}}\, \exp\Big(-\frac{t^2}{58^2\, k}\Big).
\end{equation}

We proceed to the proof of part (ii). As before, we set
$\overline{\Gcal}\coloneqq \{[n]\setminus G:G\in\Gcal\}\subseteq \binom{[n]}{n-m}$, and we observe that
$(\Fcal,\overline{\Gcal})\in\mathrm{Forbid}(n,\{k-\ell\})$ and $k-\ell\mik n-m < k\mik\frac{n}{2}$.
Thus, applying part (i) to the pair $(\Fcal,\overline{\Gcal})$ and using the fact that
$\bar{t}=\min\{k-\ell,n-m-(k-\ell)\}$, we conclude that
\begin{equation} \label{e7.5}
\frac{|\Fcal|}{\binom{n}{k}}\cdot \frac{|\Gcal|}{\binom{n}{m}} =
\frac{|\Fcal|}{\binom{n}{k}}\cdot \frac{|\overline{\Gcal}|}{\binom{n}{n-m}}
\mik 50\, \sqrt{\frac{k(n-k)m(n-m)}{n^2}}\, \exp\Big(-\frac{\bar{t}^{\, 2}}{58^2\,(n-m)}\Big). \qedhere
\end{equation}
\end{proof}
The next result supplements Theorem \ref{main} and extends the subgaussian bound \eqref{e1.main}
to a wider range of parameters $p,p'$.
\begin{prop} \label{p7.2}
Let $n$ be a positive integer, let $0<p<\frac12 < p'\mik 1-p$, and let $\ell\mik pn$ be a nonnegative integer.
Also let $\Fcal,\Gcal\subseteq \{0,1\}^n$ be two families whose cross intersections forbid\, $\ell$.
Set $t\coloneqq \min\{\ell,pn-\ell\}$ and assume that $t\meg 3$. Then we have
\begin{equation} \label{e7.6}
\mu_p(\Fcal)\, \mu_{p'}(\Gcal) \mik t\cdot\exp\Big( - \frac{t^2}{6\cdot 30^2\, pn}\Big).
\end{equation}
\end{prop}
Combining Theorem \ref{main} and Proposition \ref{p7.2}, we obtain the following corollary.
\begin{cor} \label{c7.3}
Let $n$ be a positive integer, let\, $\frac{6}{n}\mik p\mik p'\mik 1-p$, and
let $\ell\mik pn$ be a nonnegative integer. Also let $\Fcal,\Gcal\subseteq \{0,1\}^n$ be two families whose
cross intersections forbid\, $\ell$. Set $t\coloneqq \min\{\ell,pn-\ell\}$, and assume that
$t\meg 210 \sqrt{pn\ln(pn)}$. Then we have
\begin{equation} \label{e7.7}
\mu_p(\Fcal)\, \mu_{p'}(\Gcal) \mik \exp\Big( - \frac{t^2}{90^2\, pn}\Big).
\end{equation}
\end{cor}
We proceed to the proof of Proposition \ref{p7.2}.
\begin{proof}[Proof of Proposition \emph{\ref{p7.2}}]
Set $m\coloneqq \frac{t}{30}$, and observe that, by Lemma \ref{l2.3},
\begin{equation} \label{e7.8}
\max\big\{ \mu_p\big([n]^{<pn-m}\big), \mu_p\big([n]^{>pn+m}\big)\big\} \mik \exp\Big(-\frac{m^2}{6pn}\Big)
\mik \exp\Big(- \frac{t^2}{6\cdot 30^2pn}\Big).
\end{equation}
By \eqref{e7.8} and the choice of $m$, there is a nonnegative integer $i_0$ with $pn-m\mik i_0\mik pn+m$
such that, setting $\Fcal_{i_0}\coloneqq \Fcal\cap \binom{[n]}{i_0}$, we have
\begin{equation} \label{e7.9}
\mu_p(\Fcal) \mik 2\exp\Big(- \frac{t^2}{6\cdot 30^2pn}\Big)+\frac{t}{15}\, \mu_p(\Fcal_{i_0}).
\end{equation}
Set $\overline{\Gcal}\coloneqq \big\{ [n]\setminus G: G\in \Gcal\big\}$, and notice that
\begin{equation} \label{e7.10}
\mu_p(\Fcal)\, \mu_{p'}(\Gcal)=\mu_p(\Fcal)\, \mu_{1-p'}(\overline{\Gcal}) \stackrel{\eqref{e7.9}}{\mik}
\frac{t}{15}\, \mu_p(\Fcal_{i_0})\, \mu_{1-p'}(\overline{\Gcal}) + 2\exp\Big(- \frac{t^2}{6\cdot 30^2pn}\Big).
\end{equation}
Next, observe that $i_0\meg \ell$ and $(\Fcal_{i_0},\overline{\Gcal})\in\mathrm{Forbid}(n,\{i_0-\ell\})$.
Since $0<p\mik 1-p'\mik \frac12$ and $i_0-\ell\mik pn$, by Theorem \ref{main} and \eqref{e7.10}, we obtain that
\begin{equation} \label{e7.11}
\mu_p(\Fcal)\, \mu_{p'}(\Gcal) \mik  2\exp\Big(- \frac{t^2}{6\cdot 30^2pn}\Big) +
\frac{2t}{15}\,  \exp\Big(- \frac{\bar{t}^2}{58^2 pn}\Big),
\end{equation}
where $\bar{t}\coloneqq \min\{i_0-\ell, pn-i_0+\ell\}$.
\begin{claim} \label{cnew}
We have that $|\bar{t}-t|\mik m=\frac{t}{30}$.
\end{claim}
\begin{proof}[Proof of Claim \emph{\ref{cnew}}]
Suppose, towards a contradiction, that $|\bar{t}-t|>m$; that is, either $t+m<\bar{t}$ or $\bar{t}<t-m$.
We recall that $pn-m\mik i_0\mik pn+m$.

Assume, first, that $t+m<\bar{t}$. By the definition of $t$ and $\bar{t}$, we see that
\begin{enumerate}
\item[(i)] $\min\{\ell, pn-\ell\} +m < \bar{t} \mik i_0-\ell$ and
\item[(ii)] $\min\{\ell, pn-\ell\} +m < \bar{t} \mik pn-i_0+\ell$.
\end{enumerate}
If $\ell\mik pn-\ell$, then, by (ii), we obtain that $\ell+m<pn-i_0+\ell$, which is a contradiction; on the other hand,
if $pn-\ell\mik \ell$, then, by (i), we have $pn-\ell+m<i_0-\ell$ which leads, again, to a contradiction.

Next assume that $\bar{t}<t-m$. Then,
\begin{enumerate}
\item[(iii)] $\min\{i_0-\ell, pn-i_0+\ell\} +m < t\mik \ell$ and
\item[(iv)] $\min\{i_0-\ell, pn-i_0+\ell\} +m  < t \mik pn-\ell$.
\end{enumerate}
Consequently, if $i_0-\ell\mik pn-i_0+\ell$, then, by (iv), we have $i_0-\ell+m< pn-\ell$, which is a contradiction;
finally, if $pn-i_0+\ell\mik i_0-\ell$, then, by (iii), we have $pn-i_0+\ell+m<\ell$ which is also a contradiction.
\end{proof}
By \eqref{e7.11} and Claim \ref{cnew}, we conclude that
\begin{align}
\label{e7.12} \mu_p(\Fcal)\, \mu_{p'}(\Gcal) & \mik 2\exp\Big(- \frac{t^2}{6\cdot 30^2pn}\Big) +
\frac{2t}{15}\,  \exp\Big(- \frac{\bar{t}^2}{4\cdot 30^2 pn}\Big) \\
& \mik t\cdot \exp\Big( - \frac{t^2}{6\cdot 30^2\, pn}\Big). \nonumber \qedhere
\end{align}
\end{proof}


\section{Optimality} \label{sec8}

\numberwithin{equation}{section}

We proceed to discuss the optimality of the bounds obtained by Theorem \ref{main} and its extension, Corollary \ref{c7.3}.
Specifically, fix a positive integer $n$, $0<p\mik p'\mik 1-p$ and a nonnegative integer $\ell\mik pn$, set
\begin{align} \label{e8.1}
\ee_n(p,p',\ell)\coloneqq \max\!\big\{\ee>0: &
\, \mu_p(\Fcal)\,\mu_{p'}(\Gcal)\mik e^{-\ee} \text{ for every pair of} \\
& \ \ \ \ \ \ \ \ \ \
\text{nonempty families } (\Fcal,\Gcal)\in\mathrm{Forbid}(n,\{\ell\}) \big\}, \nonumber
\end{align}
and observe that our goal reduces to that of obtaining appropriate upper bounds for $\ee_n(p,p',\ell)$.
To this end, we shall additionally assume that\footnote{Notice \eqref{e8.2} slightly narrows down the regime
where the bound \eqref{e1.main} is non-trivial.}
\begin{equation} \label{e8.2}
\frac{16}{n}\mik p \ \ \ \text{ and } \ \ \ 2\sqrt{pn\ln(pn)}\mik \ell\mik pn-2\sqrt{pn\ln(pn)};
\end{equation}
we will also use the following standard lower bounds of the biased measures of the
tails of the binomial distribution (see, \textit{e.g.}, \cite[p. 115]{Ash65}).
\begin{lem} \label{l8.1}
Let $k, n$ be positive integers, and let\, $0<p\mik \frac12$. If\, $k\mik pn$, then
\begin{equation} \label{e8.3}
\mu_p\left([n]^{\mik k}\right) \meg \frac{1}{\sqrt{8n\left(1-\frac{k}{n}\right)\!\frac{k}{n}}}\,
\exp\left(-\frac{(k-pn)^2}{p(1-p)n}\right),
\end{equation}
while if\, $pn\mik k< 2pn$, then
\begin{equation} \label{e8.4}
\mu_p\left([n]^{\meg k}\right) \meg \frac{1}{\sqrt{8n\left(1-\frac{k}{n}\right)\!\frac{k}{n}}}\,
\exp\left(\frac{(k-pn)^2}{p(1-p)n}\right).
\end{equation}
\end{lem}

\subsection{The high-intersection case: $\ell\meg cpn$ for some constant $c>0$} \label{subsec8.1}

Set $\Fcal\coloneqq [n]^{<\ell}$ and $\Gcal\coloneqq \{0,1\}^n$. Then
$(\Fcal,\Gcal)\in \mathrm{Forbid}(n,\{\ell\})$ and, moreover,
\begin{align} \label{e8.5}
\mu_p(\Fcal)\, \mu_{p'}(\Gcal) =\mu_p\left([n]^{<\ell}\right)	
& \stackrel{\eqref{e8.3}}{\meg}
\frac{1}{\sqrt{8(\ell-1)(1-\frac{\ell-1}{n})}}\, \exp\left(-\frac{(pn-\ell+1)^2}{p(1-p)n}\right)\\
& \hspace{1cm} \meg \frac{1}{\sqrt{8pn}}\, \exp\left(-\frac{(pn-\ell)^2+2(pn-\ell)+1}{p(1-p)n}\right). \nonumber
\end{align}
Next, set $C\coloneqq \max\{2,\frac{pn}{\ell}\} \mik \max\{2,\frac{1}{c}\}$ and $t\coloneqq \min\{\ell,pn-\ell\}$;
notice that if $\ell\meg \frac{pn}{2}$, then $C=2$ and $t=pn-\ell$, while if $\ell<\frac{pn}{2}$, then
$C\mik \frac{1}{c}$ and $t=\ell$. By \eqref{e8.2}, \eqref{e8.5} and taking into account the previous
observations, it is easy to see that
\begin{equation} \label{e8.6}
\ee_n(p,p',\ell) \mik 4(C-1)^2\, \frac{t^2}{pn}.
\end{equation}
In particular, under \eqref{e8.2}, if $\ell\meg \frac{pn}{2}$, then
$\ee_n(p,p',\ell) \mik 4\, \frac{t^2}{pn}$.
\begin{rem} \label{r8.2}
Note that if $\mathcal{F}\coloneqq [n]^{<\ell}$, then $|A\cap B|\neq\ell$ for every
$A,B\in\mathcal{F}$ and,~by~\eqref{e8.2}~and \eqref{e8.5}, we have
$\mu_p(\mathcal{F})\meg \exp\big(- 4(C-1)^2\, \frac{t^2}{pn}\big)$, where $C= \max\{2,\frac{pn}{\ell}\}$.
Thus, Theorem~\ref{main} and Corollary \ref{c7.3} are optimal in the regime
$\ell\meg cpn$ also in the non-crossing case.
\end{rem}

\subsection{The symmetric case: $p'=1-p$} \label{subsec8.2}

We will show that, under \eqref{e8.2}, we have
\begin{equation} \label{e8.7}
\ee_n(p,1-p,\ell) \mik 4\, \frac{t^2}{pn},
\end{equation}
where, as usual, $t\coloneqq \min\{\ell,pn-\ell\}$; note that \eqref{e8.7} includes the important special
case $p=p'=\frac12$ that corresponds to the uniform probability measure on $\{0,1\}^n$.

The subcase ``$\ell\meg \frac{pn}{2}$" follows of course from \eqref{e8.6},
and so we may assume that $\ell\mik \frac{pn}{2}$. Set $\Fcal\coloneqq [n]^{>pn+\frac{\ell}{2}}$
and $\Gcal\coloneqq [n]^{\meg(1-p)n+\frac{\ell}{2}}$. Notice that $(\Fcal,\Gcal)\in\mathrm{Forbid}(n,\{\ell\})$ and
\begin{align}
\label{e8.8} \mu_{p}(\Fcal) & \meg \mu_p\big([n]^{\meg pn+\frac{\ell}{2}+1}\big) \stackrel{\eqref{e8.4}}{\meg}
\frac{1}{\sqrt{8pn}}\, \exp\left(-\frac{\ell^2}{4p(1-p)n}-\frac{\ell+1}{p(1-p)n}\right), \\
\label{e8.9} \mu_{1-p}(\Gcal) & = \mu_p\big([n]^{\mik pn-\frac{\ell}{2}}\big) \stackrel{\eqref{e8.3}}{\meg}
\frac{1}{\sqrt{8pn}}\, \exp\left(-\frac{\ell^2}{4p(1-p)n}\right).
\end{align}
Therefore, using \eqref{e8.2} and observing that in this case we have $t=\ell$, we obtain that
\begin{equation} \label{e8.10}
\mu_p(\Fcal)\, \mu_{1-p}(\Gcal)\meg \exp\left(-2\, \frac{t^2}{p(1-p)n}\right),
\end{equation}
which clearly yields \eqref{e8.7}.
\begin{rem} \label{r8.3}
It is unclear whether the subgaussian bound \eqref{e1.main} is optimal (modulo universal constants)
in the low-intersection and asymmetric case, namely, when $\ell=o(pn)$ and $p'<1-p$. The optimality of
Theorem \ref{main} in this regime is closely related\footnote{In fact, the techniques developed in this
paper show that these two problems are essentially equivalent.} to the problem of obtaining sharp estimates
of the product $\mu_p(\Fcal)\,\mu_p(\Gcal)$ of the biased measures of a pair of families
$\Fcal,\Gcal\subseteq \{0,1\}^n$ that are \emph{cross-$\ell$-intersecting}, that is, they satisfy
$|A\cap B|\meg\ell$ for every $A\in\Fcal$ and every $B\in\Gcal$. The non-crossing case, $\Fcal=\Gcal$,
is completely understood thanks to the seminal work of Ahlswede--Khachatrian \cite{AK97,AK99}
and the more recent work of Filmus \cite{Fi17}; see, also, \cite[Theorem 3.28]{Fi13} for some progress
for general cross-intersecting families.
\end{rem}


\section{Supersaturation} \label{sec9}

\numberwithin{equation}{section}

The main result in this section is a supersaturation version of Proposition \ref{p7.1},
which is the analogue of \cite[Theorem 1.14]{FR87}. To state it we need, first,
to introduce some pieces of notation. Let $n$ be a positive integer, let
$\Fcal,\Gcal\subseteq \{0,1\}^n$, and let $S\subseteq [n]$.
Given a nonnegative integer $\ell\mik n$, we set
\begin{equation} \label{e9.1}
I_\ell(\Fcal,\Gcal)\coloneqq \big\{(F,G)\in\Fcal\times\Gcal: |F\cap G|=\ell\big\}
\ \ \ \text{ and } \ \ \ i_\ell(\Fcal,\Gcal)\coloneqq |I_\ell(\Fcal,\Gcal)|
\end{equation}
and, respectively,
\begin{equation} \label{e9.2}
I_\ell(S,\Gcal)\coloneqq \big\{G\in\Gcal: |S\cap G|=\ell\big\}
\ \ \ \text{ and } \ \ \ i_\ell(S,\Gcal)\coloneqq |I_\ell(S,\Gcal)|.
\end{equation}
We have the following theorem.
\begin{thm} \label{t9.1}
Let $\ell,k,n$ be positive integers with $\ell< k\mik \frac{n}{2}$, and set $T\coloneqq \max\{\ell,k-\ell\}$.
Also let $\delta>0$ and assume that $10^5 \sqrt{k} (\ln n)^{3/2}\mik \delta \mik \min\{\ell,k-\ell\}$. Finally, set
\begin{equation} \label{e9.3}
\ee(\delta)\coloneqq \frac{\delta^4}{C\, T^2\, \ell\,\big(\ln(\frac{n}{\delta})\big)^4},
\end{equation}
where\, $C\coloneqq 2^{8}\, 58^6\, 60^4$. If\, $\Fcal\subseteq\binom{[n]}{k}$ and\,
$\Gcal\subseteq \binom{[n]}{n-k}$ satisfy
\begin{equation} \label{e9.4}
\frac{|\Fcal|}{\binom{n}{k}}\cdot\frac{|\Gcal|}{\binom{n}{n-k}}>\exp\big(-\ee(\delta)\big),
\end{equation}
then we have
\begin{equation} \label{e9.5}
\frac{i_{\ell}(\Fcal,\Gcal)}{i_{\ell}\big(\binom{[n]}{k},\binom{[n]}{n-k}\big)}> \exp\left(-\delta\right).
\end{equation}
\end{thm}
By Theorem \ref{t9.1}, we obtain the following corollary.
\begin{cor} \label{9.cor.new}
Let $\ell,k,n, T, \delta$ be as in Theorem \emph{\ref{t9.1}} and set
$\ee'(\delta)\coloneqq \frac{\delta^4}{C\, T^2\, (k-\ell)\,(\ln(\frac{n}{\delta}))^4}$,
where\, $C\coloneqq 2^{8}\, 58^6\, 60^4$. If\, $\Fcal, \Gcal\subseteq\binom{[n]}{k}$ satisfy
$\frac{|\Fcal|}{\binom{n}{k}}\cdot\frac{|\Gcal|}{\binom{n}{k}}>\exp\big(-\ee'(\delta)\big)$, then
\begin{equation} \label{en.9.1}
\frac{i_{\ell}(\Fcal,\Gcal)}{i_{\ell}\big(\binom{[n]}{k},\binom{[n]}{k}\big)}> \exp\left(-\delta\right).
\end{equation}
\end{cor}
\begin{proof}
Set $\overline{\Gcal}\coloneqq \big\{[n]\setminus G: G\in \Gcal\big\}\subseteq \binom{[n]}{n-k}$,
and notice that $i_{\ell}(\Fcal,\Gcal)=i_{k-\ell}(\Fcal,\overline{\Gcal})$ and
$i_{\ell}\big(\binom{[n]}{k},\binom{[n]}{k}\big)=i_{k-\ell}\big(\binom{[n]}{k},\binom{[n]}{n-k}\big)$.
The result follows from these observations and Theorem~\ref{t9.1} applied to $\Fcal$ and $\overline{\Gcal}$.
\end{proof}
We proceed to the proof of Theorem \ref{t9.1}.
\begin{proof}[Proof of Theorem \emph{\ref{t9.1}}]
We argue as in the proof of \cite[Theorem 1.14]{FR87} with the main new ingredients being
Theorem \ref{main} and Proposition \ref{p7.1}. For the reader's convenience we will first give a high level overview of the proof.

Our analysis is focused on the way the elements of $\Fcal$ and $\Gcal$ are correlated with arbitrary sets of size $2\ell$.
More precisely, for every $A\in\binom{[n]}{2\ell}$ we define $\Fcal_A$ and $\Gcal_A$ to be the sets $F\in \Fcal$
and $G\in \Gcal$, respectively, whose intersection with $A$ is roughly equal to~$\ell$.
We shall informally refer to these sets as ``good"\!.

The proof is then divided into three parts. In the first part, we show that there are many $A\in\binom{[n]}{2\ell}$ for which
both $\Fcal_A$ and $\Gcal_A$ are large---this is the content of Claim \ref{c9.2}. In the second part, we work towards
a contradiction  and we show that if $I_{\ell}(\Fcal,\Gcal)$ is small, then there exists $A_0\in\binom{[n]}{2\ell}$ for which
both $\Fcal_{A_0}$ and $\Gcal_{A_0}$ are large, and at the same time, there are few number of pairs $(F,G)\in\Fcal_{A_0}\times\Gcal_{A_0}$ whose intersection is of size $\ell$ and the size of its trace on $A_0$
is a specific proportion of $\ell$---this is achieved in Claim \ref{cl-8-new}.
Finally, in the third step of the proof we arrive to a contradiction. Specifically, the properties of $A_0$ imply that most of the ``good" $F\in \Fcal_{A_0}$ and $G\in \Gcal_{A_0}$ are also bad in the sense that they do not form a pair whose  intersection is of size $\ell$ and its trace on $A_0$ has size specified by the previous step; this tension is enough to derive
the contradiction.

We proceed to the details. We fix $\Fcal,\Gcal$ that satisfy~\eqref{e9.4} and we assume, towards a contradiction,
that \eqref{e9.5} does not hold true. We set $d(\Fcal)\coloneqq \frac{|\Fcal|}{\binom{n}{k}}$,
$d(\Gcal)\coloneqq\frac{|\Gcal|}{\binom{n}{n-k}}$ and
\begin{equation} \label{e9.6}
\alpha\coloneqq 116\, \sqrt{\ell\,\varepsilon(\delta)}\, \stackrel{\eqref{e9.3}}{=}
\frac{\delta^2}{(2^3\,58^2\,60^2) \, T \big(\ln(\frac{n}{\delta})\big)^2}.
\end{equation}
We also notice that for every $S\in\binom{[n]}{2\ell-\alpha}$ we have
\begin{enumerate}
\item[$\bullet$] $i_{\ell}\big( \binom{[n]}{k},\binom{[n]}{n-k}\big)=\binom{n}{k}\binom{k}{\ell}\binom{n-k}{n-k-\ell}$,
\item[$\bullet$] $i_{\ell}\big(S,\binom{[n]}{k}\big)=\binom{2\ell-\alpha}{\ell}\binom{n-2\ell+\alpha}{k-\ell}$,
\item[$\bullet$] $i_{\ell}\big(S,\binom{[n]}{n-k}\big)=\binom{2\ell-\alpha}{\ell}\binom{n-2\ell+\alpha}{n-k-\ell}$.
\end{enumerate}

Next we set
\begin{enumerate}
\item[$\bullet$] $\Scal_{\Fcal} \coloneqq \big\{ S\in\binom{[n]}{2\ell-\alpha}:
i_{\ell}(S,\Fcal)\meg K_{\Fcal}\big\}$ and
$\Scal_{\Gcal}\coloneqq \big\{ S\in\binom{[n]}{2\ell-\alpha}:
i_{\ell}(S,\Gcal)\meg K_{\Gcal}\big\}$,
\end{enumerate}
where $K_{\Fcal}\coloneqq \frac{d(\Fcal)}{2} i_{\ell}\big(S,\binom{[n]}{k}\big)$ and
$K_{\Gcal}\coloneqq \frac{d(\Gcal)}{2}i_{\ell}\big(S,\binom{[n]}{n-k}\big)$.
By \cite[Lemma 4.1]{FR87}, we~have
\begin{equation} \label{e9.7}
|\Scal_{\Fcal}| \meg \frac{d(\Fcal)}{2}\,\binom{n}{2\ell-\alpha}
\ \ \ \text{ and } \ \ \
|\Scal_{\Gcal}| \meg \frac{d(\Gcal)}{2}\, \binom{n}{2\ell-\alpha}.
\end{equation}
Moreover, for every $A\in\binom{[n]}{2\ell}$ set
\begin{enumerate}
\item[$\bullet$] $\Fcal_A\coloneqq \{F\in\Fcal:\ell\mik |F\cap A|\mik \ell+\alpha\}$
and $\Gcal_A\coloneqq \{G\in\Gcal:\ell\mik |G\cap A| \mik \ell+\alpha\}$,
\end{enumerate}
and define the family
\begin{equation} \label{e9.8}
\Acal\coloneqq \bigg\{A\in\binom{[n]}{2\ell}:
|\Fcal_A|\meg K_{\Fcal} \text{ and } |\Gcal_A| \meg K_{\Gcal}\bigg\}.
\end{equation}
\begin{claim} \label{c9.2}
We have
\begin{equation} \label{e9.9}
|\Acal| \meg \max\left\{ \frac{|\Scal_{\Fcal}|}{2\binom{2\ell}{\alpha}},
\frac{|\Scal_{\Gcal}|}{2\binom{2\ell}{\alpha}}\right\} \meg \frac{1}{4}\,
\frac{\binom{n}{2\ell-\alpha}}{\binom{2\ell}{\alpha}}\, \exp\Big(-\frac{\ee(\delta)}{2}\Big).
\end{equation}
\end{claim}
\begin{proof}[Proof of Claim \emph{\ref{c9.2}}]
We set $\Scal^*_{\Fcal}\coloneqq \{S\in\Scal_{\Fcal}:\exists S'\in\Scal_{\Gcal} \text{ with }
|S'\setminus S|\mik \alpha\}$, and we observe that
$(\Scal_{\Fcal}\setminus \Scal^*_{\Fcal},\Scal_{\Gcal})\in\mathrm{Forbid}(n,\{2\ell-2\alpha\})$.
Also note that $\varepsilon(\delta)\mik \frac{\ell}{2^{8}\, 58^6\, 60^4}$ and so, by~\eqref{e9.6},
$\alpha=\min\{2\ell-2\alpha,\alpha\}$. (In fact, $\alpha$ is significantly smaller than $\ell$.)
By part (i) of Proposition~\ref{p7.1}, we obtain
that\footnote{If $2\ell-\alpha> \frac{n}{2}$, then \eqref{e9.10} follows by applying
Proposition \ref{p7.1} to the complementary families.}
\begin{equation} \label{e9.10}
\frac{|\Scal_{\Fcal}\setminus\Scal^*_{\Fcal}|}{\binom{n}{2\ell-\alpha}}\cdot
\frac{|\Scal_{\Gcal}|}{\binom{n}{2\ell-\alpha}} \mik 50\,
\frac{(2\ell-\alpha)(n-2\ell+\alpha)}{n}\, \exp\Big(-\frac{\alpha^2}{58^2(2\ell-\alpha)}\Big).
\end{equation}
Plugging \eqref{e9.4} and \eqref{e9.7} into \eqref{e9.10}, we see that
\begin{equation} \label{e9.11}
\frac{|\Scal_{\Fcal}\setminus \Scal^*_{\Fcal}|}{|\Scal_{\Fcal}|} \mik
200\, \frac{(2\ell-\alpha)(n-2\ell+\alpha)}{n}\,
\exp\Big(\ee(\delta)-\frac{\alpha^2}{58^2(2\ell-\alpha)}\Big);
\end{equation}
since $\ee(\delta)-\frac{\alpha^2}{58^2(2\ell-\alpha)}<-\ee(\delta)$, this in turn implies that
$|\Scal^*_{\Fcal}|\meg \frac{|\Scal_{\Fcal}|}{2}$. Next, for every $S\in\Scal^*_{\Fcal}$ we select
$A\in\binom{[n]}{2\ell}$ and $S'\in\Scal_{\Gcal}$ such that $S,S'\subseteq A$, and
we observe that for each such $A$ there exist at most $\binom{2\ell}{\alpha}$ such $S$'s. Therefore,
\begin{equation} \label{e9.12}
|\Acal| \meg \frac{|\Scal_{\Fcal}|}{2\binom{2\ell}{\alpha}}.
\end{equation}
With identical arguments we verify that $|\Acal| \meg \frac{|\Scal_{\Gcal}|}{2\binom{2\ell}{\alpha}}$.
The last inequality~in~\eqref{e9.9} follows from the previous estimates, \eqref{e9.4} and \eqref{e9.7}.
The proof of Claim \ref{c9.2} is completed.
\end{proof}
Next, set
\begin{equation} \label{e9.13}
\beta\coloneqq \sqrt{2^3\, 58^2\, T\,\alpha} \stackrel{\eqref{e9.6}}{=}
\sqrt{2^4\,58^3\, T\, \sqrt{\ell\, \ee(\delta)}} \stackrel{\eqref{e9.3}}{=}
\frac{\delta}{60\ln\big(\frac{n}{\delta}\big)},
\end{equation}
and note that $116\alpha \mik\beta\mik\frac{\ell}{30}$ and $10\beta\ln\big(\frac{n}{\beta}\big)\mik\delta$.
(Indeed, $T=\max\{\ell,k-\ell\}\meg \ell\meg\alpha$ and so, by \eqref{e9.13}, we see that
$\beta\meg 116\alpha$; on the other hand, since $\delta\mik \min\{\ell,k-\ell\}\mik \ell$,
by~\eqref{e9.13} again, we obtain that $\beta\mik \frac{\ell}{30}$.) Moreover, for every $A\in\binom{[n]}{2\ell}$ set
\begin{equation} \label{e9.14}
y_{A}\coloneqq |\{(F,G)\in I_{\ell}(\Fcal,\Gcal):(F,G)\in\Fcal_{A}\times\Gcal_{A}
\text{ and } |F\cap G \cap A|=\ell-\beta\}|.
\end{equation}
For every $(F,G)\in I_{\ell}(\Fcal,\Gcal)$ we can bound the number of $A\in\Acal$ for which $(F,G)$ contributes to $y_A$ by counting the ways we can first select $F\cap G \cap A$, then $F\cap([n]\setminus G)\cap A$, then $([n]\setminus F)\cap G\cap A$, and finally $([n]\setminus F)\cap ([n]\setminus G)\cap A$. In particular, we have
\begin{equation} \label{e9.15}
\sum_{A\in\Acal} y_{A} \mik i_{\ell}(\Fcal,\Gcal) \binom{\ell}{\ell-\beta}
\sum_{0\mik i, j\mik \alpha} \binom{k-\ell}{i+\beta} \binom{n-k-\ell}{j+\beta} \binom{\ell}{i+j+\beta}.
\end{equation}
\begin{claim} \label{cl-8-new}
Given our starting assumption that \eqref{e9.5} does not hold true, we may select $A_0\in\Acal$ such that
\begin{equation} \label{e9.16}
y_{A_0} \mik
\binom{2\ell}{\ell}\binom{n-2\ell}{k-\ell} \exp\Big(-\frac{\delta}{3}\Big).
\end{equation}
\end{claim}
\begin{proof}[Proof of Claim \emph{\ref{cl-8-new}}]
First observe that
\begin{enumerate}
\item[$\bullet$] $i_{\ell} \big(\binom{[n]}{k},\binom{[n]}{n-k}\big)=
\binom{n}{k}\binom{k}{\ell}\binom{n-k}{n-k-\ell}=\binom{2\ell}{\ell}\binom{n-2\ell}{k-\ell}\binom{n}{2\ell}$.
\end{enumerate}
Next, using: (i) our starting assumption that \eqref{e9.5} does not hold true,
(ii) the fact that $\alpha+\beta\mik 2\alpha+\beta\mik \frac12\min\{\ell,k-\ell,n-k-\ell\}$, and
(iii) the fact that the function $x\mapsto \binom{y}{x}$ is increasing for $x\mik \frac{y}{2}$,
we obtain that
\begin{equation} \label{en.9.2}
\sum_{A\in\Acal}y_{A}\mik \alpha^2\binom{\ell}{\beta}\binom{k-\ell}{\alpha+\beta}
\binom{n-k-\ell}{\alpha+\beta}\binom{\ell}{2\alpha+\beta}
\binom{2\ell}{\ell}\binom{n-2\ell}{k-\ell}\binom{n}{2\ell}\exp(-\delta).
\end{equation}
Therefore, by \eqref{e9.9}, there exists $A_0\in\Acal$ such that
\begin{align} \label{en9.3}
y_{A_0} \mik \binom{2\ell}{\ell}\binom{n-2\ell}{k-\ell} & 4\alpha^2
\frac{\binom{2\ell}{\alpha}\binom{n}{2\ell}}{\binom{n}{2\ell-\alpha}}
\binom{k-\ell}{\alpha+\beta} \times \\
& \times \binom{n-k-\ell}{\alpha+\beta}
\binom{\ell}{2\alpha+\beta}\binom{\ell}{\beta}\exp\Big(\frac{\varepsilon(\delta)}{2}-\delta\Big). \nonumber
\end{align}

Now we claim that
\begin{equation} \label{e-new}
4\alpha^2\frac{\binom{2\ell}{\alpha}\binom{n}{2\ell}}{\binom{n}{2\ell-\alpha}}
\binom{k-\ell}{\alpha+\beta}\binom{n-k-\ell}{\alpha+\beta}
\binom{\ell}{2\alpha+\beta}\binom{\ell}{\beta} \mik \binom{n}{\beta}^5.
\end{equation}
To this end, notice that
\begin{equation} \label{en9.4}
\frac{\binom{2\ell}{\alpha}\binom{n}{2\ell}}{\binom{n}{2\ell-\alpha}}=\binom{n-2\ell+\alpha}{\alpha},
\end{equation}
and consequently,
\begin{equation} \label{en9.5}
\frac{\binom{2\ell}{\alpha}\binom{n}{2\ell}}{\binom{n}{2\ell-\alpha}}
\binom{k-\ell}{\alpha+\beta}\binom{n-k-\ell}{\alpha+\beta}
\binom{\ell}{2\alpha+\beta}\binom{\ell}{\beta} \mik
\binom{n}{\alpha} \binom{n}{\alpha+\beta}^2 \binom{n}{2\alpha+\beta}\binom{n}{\beta}.
\end{equation}
Moreover, by \eqref{e2.6} and the fact that $2\alpha+\beta \mik \frac{\ell}{10}\mik \frac{n}{4}$,
\begin{equation} \label{en9.6}
\frac{\binom{n}{\alpha+\beta}^2 \binom{n}{2\alpha+\beta}}{\binom{n}{\beta}^3} \mik
\Big(\frac{50}{24}\Big)^3 \, 2^{6\alpha\log_2(n)},
\end{equation}
and similarly, since $\alpha\mik \frac{\beta}{116}$,
\begin{equation} \label{en9.7}
\frac{\binom{n}{\alpha}}{\binom{n}{\beta}} \mik
\frac{25}{24}\, \sqrt{\frac{\beta}{\alpha}}\, 2^{-\frac{\beta}{2}\log_2(n)}.
\end{equation}
After observing that
\begin{equation} \label{en9.8}
\frac{\beta}{2} \log_2(n)> 2\log_2(2\alpha)+6\alpha \log_2(n)+ \frac{1}{2}\log_2\Big(\frac{\beta}{\alpha}\Big) +
\log_2\Big(\frac{25}{24}\Big) + 3\log_2\Big(\frac{50}{24}\Big),
\end{equation}
we conclude that \eqref{e-new} is satisfied.

Summing up, we see that there exists $A_0\in\Acal$ such that
\begin{equation} \label{en9.10}
y_{A_0}\mik \binom{2\ell}{\ell}\binom{n-2\ell}{k-\ell} \binom{n}{\beta}^5
\exp\Big(\frac{\varepsilon(\delta)}{2}-\delta\Big).
\end{equation}
The claim follows from this estimate together with \eqref{e2.6}, and invoking the choices of
$\varepsilon(\delta)$ and $\beta$ in \eqref{e9.3} and \eqref{e9.13}, respectively.
\end{proof}
Let $A_0\in\mathcal{A}$ be as in Claim \ref{cl-8-new}. We will show that
%
\begin{equation} \label{e9.17}
\frac{|\Fcal_{A_0}|}{y_{A_0}}>2 \ \ \ \text{ and } \ \ \ \frac{|\Gcal_{A_0}|}{y_{A_0}}>2.
\end{equation}
Indeed, by \eqref{e9.8} and the choices of $K_{\mathcal{F}}$ and $K_{\mathcal{G}}$, we have
\begin{equation} \label{en9.11}
\frac{|\Fcal_{A_0}|}{y_{A_0}}\meg \frac{\binom{2\ell-\alpha}{\ell}\binom{n-2\ell+\alpha}{k-\ell}}{2\binom{2\ell}{\ell}
\binom{n-2\ell}{k-\ell}} \, \exp\Big(\frac{\delta}{3}-\varepsilon(\delta)\Big)
\end{equation}
and
\begin{equation} \label{en9.12}
\frac{|\Gcal_{A_0}|}{y_{A_0}}\meg \frac{\binom{2\ell-\alpha}{\ell}\binom{n-2\ell+\alpha}{n-k-\ell}}{2\binom{2\ell}{\ell}
\binom{n-2\ell}{n-k-\ell}}\, \exp\Big(\frac{\delta}{3}-\varepsilon(\delta)\Big).
\end{equation}
Noticing that
\begin{equation} \label{en9.13}
\frac{\binom{2\ell-\alpha}{\ell}}{\binom{2\ell}{\ell}}\meg 2^{-\alpha}
\end{equation}
and using the previous two estimates and the choice of $\alpha$ in \eqref{e9.6}, we see that \eqref{e9.17} is satisfied.

We introduce the families
\begin{align*}
\Dcal_{\Fcal} & \coloneqq \big\{F\in\Fcal_{A_0}: \forall G\in\Gcal_{A_0}
\left(|F\cap G|\neq \ell \text{ or } |F\cap G\cap A_0|\neq \ell-\beta\right)\big\}, \\
\Dcal_{\Gcal} & \coloneqq \big\{G\in\Gcal_{A_0}: \forall F\in\Fcal_{A_0}
\left(|F\cap G|\neq \ell \text{ or } |F\cap G\cap A_0|\neq \ell-\beta\right)\big\}, \\
\Dcal_{\Fcal}^* & \coloneqq \Big\{B\subseteq A_0:
|\{F\in\Dcal_{\Fcal}:F\cap A_0=B\}|>\frac{K_{\Fcal}}{2^{2\ell+2}}\Big\}, \\
\Dcal_{\Gcal}^* & \coloneqq \Big\{B\subseteq A_0:
|\{G\in\Dcal_{\Gcal}: G\cap A_0=B\}|>\frac{K_{\Gcal}}{2^{2\ell+2}}\Big\}.
\end{align*}
Note that $|\Dcal_{\Fcal}|>\frac{K_{\Fcal}}{2}$ and $|\Dcal_{\Gcal}|>\frac{K_{\Gcal}}{2}$, and so,
there are $|D_{\Fcal}|>\frac{K_{\Fcal}}{2}$ pairs $(F,B)$ such that $F\cap A_0=B$.
For every $F\in\Dcal_{\Fcal}$ we have that $\ell \mik |F\cap A_0|\mik \ell+\alpha$; therefore,
there are at most $2^{2\ell-1}$ such choices for $B=F\cap A_0$. On the other hand,
for every such $B$ there at most $\binom{n-2\ell}{k-\ell}$ choices for $F\in D_{\Fcal}$. Hence,
\begin{equation}
\frac{K_{\Fcal}}{2}< \frac{K_{\Fcal}}{2^3}+|D_{\Fcal}^*|\binom{n-2\ell}{k-\ell},
\end{equation}
and similarly for $\mathcal{G}$. Consequently, we have
\begin{equation} \label{e9.18}
|\Dcal_{\Fcal}^*|>\frac{1}{4} \frac{K_{\Fcal}}{\binom{n-2\ell}{k-\ell}} \ \ \
\text{ and } \ \ \ |\Dcal_{\Gcal}^*|>\frac14 \frac{K_{\Gcal}}{\binom{n-2\ell}{n-k-\ell}}
\end{equation}
that implies that
\begin{equation} \label{e9.19}
|\Dcal_{\Fcal}^*|\cdot |\Dcal_{\Gcal}^*| \meg 2^{4\ell+1}\,
\exp\left(-2(\varepsilon(\delta)+2\alpha)\right) >
2^{4\ell+1}\exp\left(-\frac{\beta^2}{58^2\ell}\right).
\end{equation}
By Theorem \ref{main} applied for ``$p=p'=\frac12$", there exist
$B_1\in \Dcal_{\Fcal}^*$ and $B_2\in \Dcal_{\Gcal}^*$ such that $\left|B_1\cap B_2\right|=\ell-\beta$.
Next, set $x\coloneqq |B_1|-\ell$, $\Xcal\coloneqq \{F\setminus A_0:F\in \Dcal_\Fcal \text{ and } F\cap A_0=B_1\}$ and
$y\coloneqq |B_2|-\ell$, $\Ycal\coloneqq \{G\setminus A_0:G\in \Dcal_\Gcal \text{ and } G\cap A_0=B_2\}$.
Observe that $0\mik x,y\mik\alpha$ and
\begin{equation} \label{e9.20}
|\Xcal|\meg \frac{K_{\Fcal}}{2^{2\ell+2}} \ \ \ \text{ and } \ \ \
|\Ycal|\meg \frac{K_{\Gcal}}{2^{2\ell+2}}.
\end{equation}
Therefore, by our assumption on $\delta$ and the choice of $\alpha$ and $\beta$,
\begin{align}
\label{e9.21} |\Xcal|\cdot |\Ycal| & \meg \binom{n-2\ell}{k-\ell-x}\cdot
\binom{n-2\ell}{n-k-\ell-y}\cdot \frac{1}{2^{10}}\cdot \frac{1}{4\ell}\cdot e^\alpha\cdot
\exp\big(-2(\ee(\delta)+2\alpha)\big) \\
& > \binom{n-2\ell}{k-\ell-x}\cdot \binom{n-2\ell}{n-k-\ell-y}
\cdot 50n\cdot  \exp\left(-\frac{\beta^2}{58^2(k-\ell-x)}\right). \nonumber
\end{align}
By part (i) of Proposition \ref{p7.1}, there exist $H_1\in\Xcal$, $H_2\in \Ycal$ such that $\left|H_1\cap H_2\right|=\beta$.
It follows that $(B_1\cup H_1,B_2\cup H_2)\in (\Fcal_{A_0}\times\Gcal_{A_0})\cap(\Dcal_{\Fcal}\times\Dcal_{\Gcal})$,
which clearly leads to a contradiction.
\end{proof}

\subsection*{Acknowledgment}

The authors would like to thank the anonymous referee for numerous comments, remarks and suggestions
that helped us improve the exposition.

The research was supported by the Hellenic Foundation for Research and Innovation
(H.F.R.I.) under the “2nd Call for H.F.R.I. Research Projects to support
Faculty Members \& Researchers” (Project Number: HFRI-FM20-02717).


\end{document}